\documentclass[11pt]{amsart}

\usepackage{amsmath,amssymb}
\usepackage{amsthm}
\usepackage[alphabetic]{amsrefs}
\usepackage{indentfirst}
\usepackage{mathrsfs}
\usepackage{graphicx}
\usepackage{hyperref}
\usepackage[margin=1.1in]{geometry}
\usepackage{xcolor}
\allowdisplaybreaks

\IfFileExists{mathabx.sty}%
{\DeclareFontFamily{U}{mathx}{\hyphenchar\font45}%
  \DeclareFontShape{U}{mathx}{m}{n}{<->mathx10}{}%
  \DeclareSymbolFont{mathx}{U}{mathx}{m}{n}%
  \DeclareFontSubstitution{U}{mathx}{m}{n}%
  \DeclareMathAccent{\widebar}{0}{mathx}{"73}%
}{
 \PackageWarning{mathabx}{}
 \newcommand{\widebar}[1]{\overline{#1}}
}

\newcommand{\wb}[1]{\widebar{#1}}

\newcommand{\wt}[1]{\widetilde{#1}}
\newcommand{\wh}[1]{\widehat{#1}}

\newcommand{\rd}{\mathrm{d}}
\newcommand{\ud}{\,\mathrm{d}}

\DeclareMathOperator{\divop}{div}
\DeclareMathOperator{\tr}{tr}

\newcommand{\TT}{\mathrm{T}}

\newcommand{\be}{\begin{equation}}
\newcommand{\ee}{\end{equation}}
\newcommand{\bi}{\begin{itemize}}
\newcommand{\ei}{\end{itemize}}
\newcommand{\br}{\begin{eqnarray}}
\newcommand{\er}{\end{eqnarray}}

\newcommand{\expE}{\mathbb{E}}

\newcommand{\abs}[1]{\lvert #1 \rvert}
\newcommand{\norm}[1]{\lVert #1 \rVert}

\newcommand{\avg}[1]{\langle #1 \rangle}

\newcommand{\commentout}[1]{}

\newcommand{\RR}{\mathbb{R}}

\newtheorem{theo}{Theorem}[section]
\newtheorem{prop}[theo]{Proposition}

\newtheorem{lem}[theo]{Lemma}
\newtheorem{cor}[theo]{Corollary}

\newtheorem{rmk}[theo]{Remark}

\newcommand{\Rm}{{\mathbb R}}

\newcommand{\Em}{{\mathbb E}}
\newcommand{\Pm}{{\mathbb P}}

\def\XXint#1#2#3{{\setbox0=\hbox{$#1{#2#3}{\int}$}
\vcenter{\hbox{$#2#3$}}\kern-.5\wd0}}

\numberwithin{equation}{section}

\begin{document}

\title{Reactive trajectories and the transition path process}

\author{Jianfeng Lu} 

\address{Department of Mathematics and Department of Physics \\
  Duke University, Box 90320 \\
  Durham, NC 27708, USA}

\email{jianfeng@math.duke.edu} 

\author{James Nolen}
\address{Mathematics Department \\
  Duke University, Box 90320 \\
  Durham, NC 27708, USA}
\email{nolen@math.duke.edu}

\date{\today}

\thanks{We are grateful to Weinan E, Jonathan Mattingly, and Eric
  Vanden-Eijnden for helpful discussions.}

\maketitle

\begin{abstract}
  We study the trajectories of a solution $X_t$ to an It\^o stochastic
  differential equation in $\Rm^d$, as the process passes between two
  disjoint open sets, $A$ and $B$. These segments of the trajectory
  are called transition paths or reactive trajectories, and they are
  of interest in the study of chemical reactions and thermally
  activated processes. In that context, the sets $A$ and $B$ represent
  reactant and product states. Our main results describe the
  probability law of these transition paths in terms of a transition
  path process $Y_t$, which is a strong solution to an auxiliary SDE
  having a singular drift term.  We also show that statistics of the
  transition path process may be recovered by empirical sampling of
  the original process $X_t$. As an application of these ideas, we
  prove various representation formulas for statistics of the
  transition paths. We also identify the density and current of
  transition paths. Our results fit into the framework of the
  transition path theory by E and Vanden-Eijnden.
\end{abstract}

\section{Introduction}

In this article we study solutions $X_t \in \Rm^d$ of the It\^o
stochastic differential equation
\begin{equation}\label{XSDE}
  \rd X_t = b(X_t) \ud t + \sqrt{2}\,\sigma(X_t) \ud W_t,
\end{equation}
where $(W_t, \mathcal{F}^W_t)$ is a standard Brownian motion in
$\Rm^d$, defined on a probability space $(\Omega, \mathcal{F},
\Pm)$. This diffusion process in $\Rm^d$ has generator
\begin{equation*}
  L u = \tr(a \nabla^2 u) + b \cdot \nabla u,
\end{equation*}
where $a:= \sigma \sigma^{\TT}$ is a symmetric matrix. We suppose that $a(x)$ is
smooth, uniformly positive definite, and bounded:
\[
\lambda |\xi|^2 \leq \avg{\xi, a(x) \xi} \leq \Lambda |\xi|^2, \quad
\forall\,\xi \in \Rm^d, \quad \forall \;x \in \Rm^d
\]
holds for some $\Lambda > \lambda > 0$.  We suppose the vector field
$b(x)$ is smooth and satisfies conditions that guarantee the
ergodicity of the Markov process $X_t$ and the existence of a unique
invariant probability distribution $\rho(x) > 0$ satisfying the adjoint equation
\begin{equation}\label{rhodef}
  L^* \rho = ( a_{ij}(x) \rho(x))_{x_i x_j} - \nabla \cdot (b(x)
  \rho(x)) = 0.
\end{equation}
For example, this will be the case if
\[
\limsup_{ m \to +\infty} \;\sup_{|x| = m} x \cdot b(x)  < -r
\] 
with $r > 1 + (d/2)$  \cite{Ver97}.

Suppose that $A, B \subset \Rm^d$ are two bounded open sets with
smooth boundary and such that $\wb{A}$ and $\wb{B}$ are
disjoint. Because the process is ergodic, $X_t$ will visit both $A$
and $B$ infinitely often. Inspired by the transition path theory
  developed by E and Vanden-Eijnden \cites{EVa:06, MeScVa:06} (see
  also the review article \cite{EVa:10}), our main interest is in
those segments of the trajectory $t \mapsto X_t$ which pass from $A$
to $B$. These transition paths and are defined precisely as
follows. First, for $k \geq 0$, define the hitting times
$\tau_{A,k}^+$ and $\tau_{B,k}^+$ inductively by
\begin{align*}
  & \tau_{A,0}^+ = \inf \{ t \geq 0 \mid X_t \in \wb A\},  \\
  & \tau_{B,0}^+ = \inf \{ t > \tau_{A,0}^+ \mid X_t \in \wb B\}, \\
  \intertext{and for $k \geq 0$,}
  & \tau_{A,k+1}^+ = \inf \{ t > \tau_{B,k}^+ \mid X_t \in \wb A\}, \\
  & \tau_{B,k+1}^+ = \inf \{ t > \tau_{A,k+1}^+ \mid X_t \in \wb B\}. \\
  \intertext{We will call these the \textbf{entrance times}. Then
    define the \textbf{exit times}}
  & \tau_{A,k}^- = \sup \{ t < \tau_{B,k}^+ \mid X_t \in \wb A\}, \\
  & \tau_{B,k}^- = \sup \{ t < \tau_{A,k+1}^+ \mid X_t \in \wb B\}.
\end{align*}
These times are all finite with probability one, and $\tau_{A,k}^+
\leq \tau_{A,k}^- < \tau_{B,k}^+ \leq \tau_{B,k}^- < \tau_{A,k+1}^+$
for all $k \geq 0$ (see Figure \ref{fig:transpath}). If $t \in
[\tau_{A,k}^-, \tau_{B,k}^+]$ for some $k$, we say that the path $X_t$
is $A \to B$ \textbf{reactive}. Let $\Theta = (\wb{A \cup B})^C$, and
hence $\partial\, \Theta = \partial A \cup \partial B$. For $k \in
\mathbb{N}$, the continuous process $Y^k:[0,\infty) \to \wb{\Theta}$
defined by
\begin{equation}\label{nthreactive}
  Y_t^k = X_{(t + \tau_{A,k}^-) \wedge \tau_{B,k}^+}
\end{equation} 
is the $k^{\text{th}}$ $A \to B$ \textbf{reactive trajectory} or
\textbf{transition path}. Observe that $Y_0^k = X_{\tau_{A,k}^-}
\in \partial A$, and that $Y^k_t = X_{\tau_{B,k}^+} \in \partial B$
for all $t \geq \tau_{B,k}^+ - \tau_{A,k}^-$, and that $Y^k_t \in
\Theta$ for all $t \in (0,\tau_{B,k}^+ - \tau_{A,k}^-)$.

\begin{figure}[htp]
  \label{fig:transpath}
  \includegraphics[height = 2.5in]{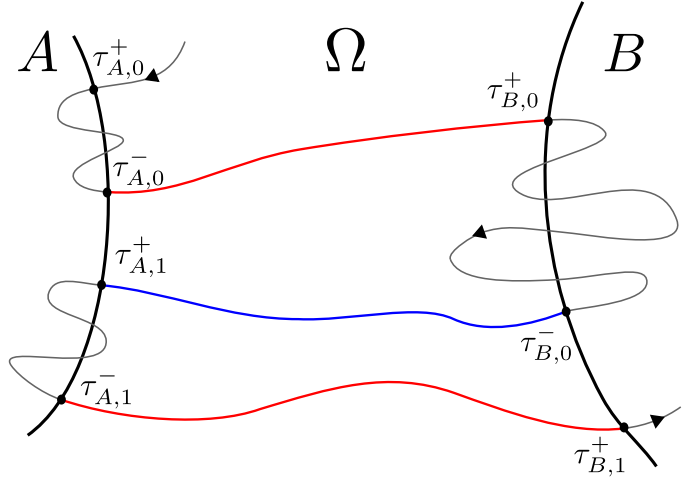}
  \caption{Illustration of a trajectory with entrance and exit
    times. The transition path from $A$ to $B$ is marked in red.}
\end{figure}
 
Our main results describe the probability law of these transition
paths in terms of a \textbf{transition path process}, which is a strong
solution to an auxiliary stochastic differential equation. In
particular, empirical samples of the reactive portions of $X_t$ may be
regarded as sampling from the transition path process. The motivation
comes from the study of chemical reactions and thermally activated
processes where understanding these reactive trajectories are crucial
\cites{DellagoBolhuisGeissler:02, BolhuisChandlerDellagoGeissler:02}.
In these applications, the domains $A$ and $B$ are usually chosen as
regions in configurational space corresponding to reactant and product
states. Mathematically, our results fit into the framework of the
  transition path theory \cites{EVa:10, EVa:06, MeScVa:06}.

Having identified the transition path process, we can compute
statistics of the transition paths by sampling directly from the
transition path SDE, rather than using acceptance/rejection methods or
very long-time integration on the original SDE. Of course, this
assumes knowledge of the committor function, which is non-trivial. In any
case, our results might be used to analyze methods of sampling
reactive trajectories.

We will now describe our main results and their relation to other
works. Proofs are deferred to later sections.

\subsection{The transition path process}

Our definition of the transition path process is motivated by the Doob
$h$-transform as follows. Let $\tau_A$ and $\tau_B$ denote the first
hitting time of $X_t$ to the sets $\wb{A}$ and $\wb{B}$, respectively:
\begin{equation}
  \begin{aligned}
    \tau_A & = \inf \left\{ t \geq 0 \mid X_t \in \wb{A} \right\},  \\
    \tau_B & = \inf \left\{ t \geq 0 \mid X_t \in \wb{B} \right\}.
  \end{aligned}
\end{equation}
Let $q(x)\geq 0$ be the \textbf{forward committor function}: 
\begin{equation}\label{committorDef}
q(x) = \Pm(\tau_A > \tau_B \mid X_0 = x),  
\end{equation}
which satisfies $L q(x) = 0$ for $x \in \Theta = (\wb{A \cup B})^C$ and 
\begin{equation}\label{committorBC}
  q(x)  = 
  \begin{cases}
    0, & x \in \wb A, \\ 
    1, & x \in \wb B. 
  \end{cases} 
\end{equation}
By the maximum principle, $q(x) > 0$ for all $x \in \Theta$. By the
Hopf lemma we also have
\begin{equation}\label{hopf}
  \sup_{x \in \partial A} \wh{n}(x) \cdot \nabla q(x) < 0, \qquad 
  \inf_{x \in \partial B} \wh{n}(x) \cdot \nabla q(x) > 0, 
\end{equation}
where $\wh{n}(x)$ will denote the unit normal exterior to $\Theta$
(pointing into $A$ and $B$). For $x \in \Theta$, consider the stopped
process $X_{t \wedge \tau_A \wedge \tau_B}$ with $X_0 = x$, and let
$\mathcal{P}_x$ denote the corresponding measure on $\mathcal{X} =
C([0,\infty))$:
\[
\mathcal{P}_x(U) = \Pm( X \in U \;|\; X_0 = x), \quad \forall\; U \in
\mathcal{B}
\]
where $\mathcal{B}$ is the Borel $\sigma$-algebra on $\mathcal{X}$. If
$\Lambda_{AB}$ denotes the event that $\tau_A > \tau_B$, the measure
$\mathcal{Q}_x^q$ on $(\mathcal{X},\mathcal{B})$ defined by
\[
\dfrac{\rd \mathcal{Q}_x^q}{\rd \mathcal{P}_x} =
\dfrac{\mathbb{I}_{\Lambda_{AB}}}{\mathcal{P}_x(\Lambda_{AB})} =
\dfrac{\mathbb{I}_{\Lambda_{AB}}}{q(x)}
\]
is absolutely continuous with respect to $\mathcal{P}_x$, if $x \in
\Theta$.  By the Doob $h$-transform (see e.g.~\cite{Day:92},
\cite{Pinsky:95}*{Theorem 7.2.2}), we know that $\mathcal{Q}_x^q$
defines a diffusion process $Y_t$ on $C([0,\infty))$ with generator:
\begin{equation}\label{Lconditioned} 
  L^q f = \frac{1}{q} L(qf) = 
  \tr(a \nabla^2 f) + (b \cdot \nabla f) + \frac{2 a \nabla q}{q}
  \cdot \nabla f = Lf + \frac{2a \nabla q}{q} \cdot 
  \nabla f.
\end{equation}
So, the effect of conditioning on the event $\tau_B < \tau_A$ is to
introduce an additional drift term. 

This observation suggests that the $A\to B$ reactive trajectories
should have the same law as a solution to the SDE
\begin{equation}\label{TPode}
  \rd Y_t = \Bigl(b(Y_t) +  \dfrac{2a(Y_t)\nabla q(Y_t) }{q(Y_t)}\Bigr) 
  \ud t + \sqrt{2}\, \sigma(Y_t) \ud \wh{W}_t,
\end{equation}
originating at a point $Y_0 = y_0 \in \partial A$ and terminating at a
point in $\partial B$. While the SDE \eqref{TPode} admits strong
solutions for $y_0 \in \Theta$ since $q(x) > 0$ in $\Theta$, the drift
term becomes singular at the boundary of $A$, where $q$ vanishes. Our
first result is the following theorem which shows that there is still a
unique strong solution to this SDE even for initial condition lying in
$\partial A$. For convenience, let us define the vector field 
\begin{equation}\label{Kvecdef}
  K(y) = \Bigl(b(y) + \frac{2 a(y)\nabla q(y)}{q(y)}\Bigr).
\end{equation}

\begin{theo}\label{theo:uniqueYAB} 
  Let $(\wh{W},\mathcal{F}^{\wh{W}}_t)$ be a standard Brownian motion
  in $\Rm^d$, defined on a probability space $(\wh{\Omega},
  \wh{\mathcal{F}},\mathbb{Q})$. Let $\xi:\wh{\Omega} \to \wb{\Theta}$
  be a random variable defined on the same probability space and
  independent of $\wh{W}$. There is a unique, continuous process
  $Y_t:[0,\infty) \to \wb{\Theta}$ which is adapted to the augmented
  filtration $\wh{\mathcal{F}}_t$ and satisfying the following,
  $\mathbb{Q}$-almost surely:
  \begin{align}
    Y_t = \xi + \int_0^{t \wedge \tau_B} K(Y_s) \ud s + \int_0^{t
      \wedge \tau_B} \sqrt{2}\, \sigma(Y_s) \ud \wh{W}_s, \quad t \geq
    0\label{TPode1}
  \end{align}
  where
  \[
  \tau_B = \inf \{ t > 0 \mid Y_t \in \wb B\}.
  \]
Moreover, $Y_t \not \in \wb{A}$ for all $t > 0$.
\end{theo}

The augmented filtration is defined in the usual way,
$\wh{\mathcal{F}}_t$ being the $\sigma$-algebra generated by
$\mathcal{F}_t^{\wh{W}}$, $Y_0$, and the appropriate collection of
null sets so that $\wh{\mathcal{F}}_t$ is both left- and right-
continuous.  We will use $\wh{\expE}$ to denote expectation with
respect to the probability measure $\mathbb{Q}$.

Observe that if $d = 1$, $\sigma = 1/\sqrt{2}$ is constant, and $b
\equiv 0$, then $q(x)$ is a linear function, and (\ref{TPode})
corresponds to a Bessel process of dimension 3. For example, if $A =
(-\infty,0)$, $B = (1,\infty)$, we have
\[
\rd Y_t = \frac{1}{Y_t} \ud t + \ud\wh{W}_t,
\]
and the function $Z_t = (Y_t)^2$ satisfies the degenerate diffusion
equation 
\begin{equation}\label{Z1d}
  \rd Z_t = 3 \ud t + 2 \sqrt{Z_t} \ud \wh{W}_t. 
\end{equation}
In this simple case, existence and uniqueness of a strong solution
starting at $Y_0 = 0$ can be shown using arguments involving Brownian
local time (see \cites{ReYo:04, KaSh:91}). However, those arguments
are not applicable to the more general setting we consider here. The
work most closely related to Theorem~\ref{theo:uniqueYAB} in a higher
dimensional setting may be that of DeBlaissie \cite{Deb:04} who proved
pathwise uniqueness for certain SDEs having diffusion coefficients
that degenerate like $\sqrt{d(Z_t)}$ where $d(z)$ is the distance to
the domain boundary (as in \eqref{Z1d}). In an earlier work, Athreya,
Barlow, Bass, and Perkins \cite{ABBP:2002} proved uniqueness for the
martingale problem associated with a similarly degenerate diffusion in
a positive orthant in $\Rm^d$.  Nevertheless, those analyses do not
apply to the case \eqref{TPode} considered here.

Our next result is the following theorem which shows that the law of
the reactive trajectories is that of the process $Y_t$ with
appropriate initial condition.  For this reason, we will call the
process $Y_t$ the \textbf{transition path process}.

\begin{theo}\label{theo:XtauAmin} 
  Let $X_t$ satisfy the SDE (\ref{XSDE}).  Let $Y^k$ denote the
  $k^{\text{th}}$ $A \to B$ reactive trajectory defined by
  \eqref{nthreactive}. Let $Y$ be defined as in Theorem
  \ref{theo:uniqueYAB}. Then for any bounded and continuous functional
  $F:C([0,\infty)) \to \Rm$, we have
  \[
  \expE[F(Y^k)] = \wh{\expE}\left[F(Y) \mid Y_0 \sim X_{\tau_{A,k}^-}
  \right].
  \]
\end{theo}
The processes $X_t$ and $Y^k_t$ are defined on a probability space
that is different from the one on which $Y_t$ is defined. The notation
$Y_0 \sim X_{\tau_{A,k}^-}$ used in Theorem~\ref{theo:XtauAmin} means
that $Y_0$ has the same law as $X_{\tau_{A,k}^-}$, meaning
$\mathbb{Q}(Y_0 \in U) = \Pm(X_{\tau_{A,k}^-} \in U)$ for any Borel
set $U \subset \Rm^d$.

%\bigskip

\subsection{Reactive exit and entrance distributions}

The distribution of the random points $X_{\tau_{A,k}^-}$ will depend
in the initial condition $X_0$. From the point of view of sampling the
transition paths, however, there is a very natural distribution to
consider for $Y_0$. To motivate this distribution formally, let $h >
0$ and consider the regularized hitting times
\begin{align}
  \tau_{A,h} & = \inf \left\{ t \geq h \mid X_t \in \wb{A} \right\}  \\
  \tau_{B,h} & = \inf \left\{ t \geq h \mid X_t \in \wb{B} \right\},
\end{align}
where $X_t$ satisfies (\ref{XSDE}). Then define
\[
q_h(x) = \Pm(\tau_{A,h} > \tau_{B,h} \mid X_0 = x).
\]
This is the probability that at some time $s \in [0,h]$,
the path $X_t$ starting from $x \in \partial A$ becomes a transition path, not
returning to $\bar A$ before hitting $\bar B$. With this in mind, the
quantity
\[
\eta_{A,h}(x) = h^{-1} \rho(x) \Pm(\tau_{A,h} > \tau_{B,h} \mid X_0 =
x) = h^{-1} \rho(x) q_h(x),
\]
may be interpreted as a rate at which transition paths exit $A$, when
the system is in equilibrium. Therefore, a natural choice for an
initial distribution for $Y_0 \in \partial A$ is:
\[
\eta_{A}(x) = \lim_{h \to 0} \eta_{A,h}.
\]

By the Markov property, we have
\begin{equation}
  q_h(x) =  \int_{\Rm^d} \Pm(\tau_A > \tau_B\;|\; X_0 = y) \rho(h,x,y) \ud y =
  \expE[q(X_h)\mid X_0 = x]
\end{equation}
where $\rho(t,x,\cdot)$ is the density for $X_t$, given $X_0 =
x$. Therefore, for any $x \in \partial A$ we have
\[
\lim_{h \to 0} h^{-1} q_h(x) = \lim_{h \to 0} h^{-1}\expE[q(X_h) -
q(X_0)\;|\;X_0 = x] = L q(x),
\]
in the sense of distributions, although $q$ is not $C^2$ on $\partial
\Theta = \partial A \cup \partial B$. Hence $\eta_{A,h}(x) \to
\eta_A(x) = \rho(x)Lq(x)$ for $x \in \partial A$. The distribution
$Lq$ is supported on $\partial \Theta$. If $\phi$ is a smooth test
function supported on a set $B_r(x)$, a small neighborhood of $x
\in \partial A$, then we have
\begin{equation*}
  \begin{aligned}
    \avg{L q, \phi} & = \int_{\Rm^d} q(x) L^{\ast} \phi(x) \ud x \\
    & = \int_{B_r(x) \cap \Theta} L q(x) \phi(x) \ud x +
    \int_{(\partial A) \cap B_r(x)} \bigl( q \wh{n} \cdot \divop (a
    \phi) - (\wh{n} \cdot a \nabla q) \phi + q \wh{n} \cdot b\, \phi
    \bigr) \ud \sigma_A(x)
  \end{aligned}
\end{equation*}
where $\wh{n}(x)$ is the unit normal vector exterior to $\Theta$, and
$\ud\sigma_A$ is the surface measure on $\partial A$. Since $q = 0$
on $\partial A$ and $Lq = 0$ on $\Theta$, this implies,
\begin{equation*}
  \avg{L q, \phi} = - \int_{(\partial A) \cap B_r(x)} 
  \phi\, \wh{n} \cdot a \nabla q \ud \sigma_A(x).
\end{equation*}
That is (after a similar calculation for points on $\partial B$), 
\begin{equation}\label{eq:Lq}
  L q(x) = - \wh{n}(x) \cdot a(x)\nabla q(x) \ud \sigma_A(x) -
  \wh{n}(x) \cdot a(x) \nabla q(x) \ud \sigma_B(x),
\end{equation}
in the sense of distributions. Restricting on $\partial A$, we get 
\begin{equation}\label{etaArestrict}
  \eta_A = - \rho(x) \wh{n}(x) \cdot a(x)\nabla q(x) \ud \sigma_A(x). 
\end{equation}
By switching the role of $A$ and $B$ in the above discussion, it is
also natural to define a measure on $\partial B$ as
\begin{equation}
  \eta_B = \rho(x) \wh{n}(x) \cdot a(x) \nabla q(x) \ud \sigma_B(x).
\end{equation}
Note that $1 - q$ gives the forward committor function for the
transition from $B$ to $A$ and that $L q(x) = \eta_A(\rd x) -
\eta_B(\rd x)$. Although the distributions $\eta_A$ and $\eta_B$ are
positive (by (\ref{hopf})), they need not be probability
distributions. Nevertheless, the mass of the two measures is the same.
\begin{lem}\label{lem:sameMass} The measures $\eta_A$ and $\eta_B$
  satisfy $\eta_A(\partial A) = \eta_B(\partial B)$. That is,
  \begin{equation}\label{eq:normetaAetaB}
    \int_{\partial A} \rho(x) \wh{n}(x)\cdot a(x) \nabla q(x) \ud \sigma_A(x)
    + \int_{\partial B} \rho(x) \wh{n}(x)\cdot a(x) \nabla q(x) \ud \sigma_B(x) 
    = 0.
  \end{equation}
\end{lem}

This computation motivates us to define
\begin{align}
  \label{eq:etaA-} & \eta_A^-(\rd x) = \frac{1}{\nu} \eta_A(\rd x) =
  - \frac{1}{\nu} \rho(x) \wh{n}(x) \cdot a(x) \nabla q(x) \ud \sigma_A(x), \\
  \label{eq:etaB-} & \eta_B^-(\rd x) = \frac{1}{\nu} \eta_B(\rd x) =
  \frac{1}{\nu} \rho(x) \wh{n}(x) \cdot a(x) \nabla q(x) \ud
  \sigma_B(x),
\end{align} 
We call these distributions the \textbf{reactive exit distribution} on
$\partial A$ and on $\partial B$, respectively. The constant $\nu$ is
a normalizing constant so that $\eta_A^-$ and $\eta_B^-$ define
probability measures on $\partial A$ and $\partial B$. By Lemma
\ref{lem:sameMass}, the normalizing constant is the same for both
measures. Our next result relates the reactive exit distribution on
$\partial A$ to the \textbf{empirical reactive exit distribution} on
$\partial A$, defined by
\begin{equation}\label{empexit}
  \mu_{A,N}^- = \frac{1}{N} \sum_{k=0}^{N-1} \delta_{X_{\tau_{A,k}^-}}(x).
\end{equation}

\begin{prop} \label{prop:exitdistr}
Let $\mu_{A,N}^-$ be the empirical reactive exit
  distribution on $\partial A$ defined by \eqref{empexit}. Then
  $\mu_{A,N}^-$ converges weakly to $\eta_A^-$ as $N \to \infty$. That
  is, for any continuous and bounded $f:\partial A \to \Rm$
  \[
  \lim_{N \to \infty} \int_{\partial A} f(x) \ud \mu_{A,N}^-(x) =
  \int_{\partial A} f(x) \ud \eta_A^-(x)
  \]
  holds $\Pm$-almost surely.
\end{prop}

A similar statement holds for the reactive exit distribution on
$\partial B$ and the empirical distribution of the points
$X_{\tau_{B,k}^-}$.  The reactive exit distribution $\eta_A^-(\rd x)$ is related to the equilibrium measure $e_{A, B}(\rd x)$
in the potential theory for diffusion processes \cite{Sznitman:98,
  BoEcGaKl:04, BoGaKl:05}. In fact, the committor function $q$ is
known as the equilibrium potential in those works, and the equilibrium
measure $e_{A, B}(\rd x)$ is given by $Lq$ restricted on $\partial
A$. Specifically, we have
\begin{equation} \label{etaAeAB}
  \eta_A^-(\rd x) = \frac{1}{\nu} \rho(x) e_{A, B}(\rd x). 
\end{equation}
To the best of our knowledge, Proposition~\ref{prop:exitdistr} for the
first time characterizes the equilibrium measure from a dynamic
perspective.

\bigskip

We also identify the limit of the \textbf{empirical reactive entrance
  distribution} on $\partial B$, defined as
\begin{equation}\label{empentrance}
  \mu_{B, N}^{+} = \frac{1}{N} \sum_{k=0}^{N-1} \delta_{X_{\tau_{B, k}^+}}(x). 
\end{equation}
To describe its limit as $N \to \infty$, let us denote by $\wt{L}$ the
adjoint of $L$ in $L^2(\Rm^d, \rho(x) \rd x)$, given by
\begin{equation}\label{LtildeDef}
  \wt{L} u = - b \cdot \nabla u + \frac{2}{\rho} \divop( a \rho) \cdot \nabla u
  + \tr(a \nabla^2 u). 
\end{equation}
This corresponds to the generator of the time-reversed process $t
\mapsto X_{T - t}$ \cite{HaussmannPardoux:86}. Note that $\wt{L} = L$
if the SDE \eqref{XSDE} is reversible, \textit{i.e.} $L$ is
self-adjoint in $L^2(\Rm^d, \rho(x) \ud x)$. In addition to the
forward committor function $q(x)$ (recall \eqref{committorDef}), we also
define the \textbf{backward committor function} $\wt{q}(x)$ to be the
unique solution of
\[
\wt{L} \wt{q} = 0, \quad x \in \Theta
\]
with boundary condition
\[
\wt{q}(x) = 
\begin{cases}
  1, & x \in \partial A \\
  0, & x \in \partial B.
\end{cases}
\]
In terms of $\wt{q}$, we define the \textbf{reactive entrance
  distribution} on $\partial B$ as
\begin{equation}
  \label{eq:etaB+} \eta_B^+(\rd x) = - \frac{1}{\nu} 
  \rho(x) \wh{n}(x) \cdot a(x)
  \nabla \wt{q}(x) \ud \sigma_B(x)
\end{equation}
and analogously the reactive entrance distribution on $\partial A$
\begin{equation}
  \label{eq:etaA+} \eta_A^+(\rd x) = \frac{1}{\nu} 
  \rho(x) \wh{n}(x) \cdot a(x) \nabla \wt{q}(x)
  \ud \sigma_A(x).
\end{equation}
Again, $\nu$ is a normalizing constant so that these are probability
measures; $\nu$ is the same as the constant in \eqref{eq:etaA-}. The
following proposition justifies the definition of the reactive
entrance distribution.

\begin{prop}\label{prop:etaB+} Let $\mu_{B,N}^+$ be the empirical
  reactive entrance distribution on $\partial B$ defined by
  \eqref{empentrance}. Then $\mu_{B,N}^+$ converges weakly to
  $\eta_B^+$ as $N \to \infty$. That is, for any continuous and
  bounded $f:\partial B \to \Rm$
  \[
  \lim_{N \to \infty} \int_{\partial B} f(x) \ud \mu_{B,N}^+(x) =
  \int_{\partial B} f(x) \ud \eta_B^+(x)
  \]
  holds $\Pm$-almost surely.  
\end{prop}

A similar statement holds for the reactive entrance distribution on
$\partial A$ and the empirical distribution of the points $X_{\tau_{A,
    k}^+}$. 

\begin{rmk}
  If the SDE \eqref{XSDE} is reversible, we have $\wt{q} = 1 - q$, and
  hence $\eta_A^+(\rd x) = \eta_A^-(\rd x)$ and $\eta_B^+(\rd x) =
  \eta_B^-(\rd x)$.
\end{rmk}

In view of Proposition \ref{prop:exitdistr}, $\eta_A^-$ is a natural
choice for the distribution of $Y_0$. With this choice, the transition
path process $Y_t$ characterizes the empirical distribution of $A \to
B$ reactive trajectories, as the next theorem shows:

\begin{theo} \label{theo:empirical} Let $X_t$ satisfy the SDE
  \eqref{XSDE}.  Let $Y^k$ denote the $k^{\text{th}}$ $A \to B$
  reactive trajectory defined by \eqref{nthreactive}. Let $Y$ be the
  unique process defined by Theorem~\ref{theo:uniqueYAB} with initial
  distribution $Y_0 \sim \eta_A^-(\rd x)$ on $\partial A$ defined by
  \eqref{eq:etaA-}, and let $\mathcal{Q}_{\eta_A^-}$ denote the law of
  this process on $\mathcal{X} = C([0,\infty))$. Then for any $F \in
  L^1(\mathcal{X},\mathcal{B},\mathcal{Q}_{\eta_A^-})$, the limit
  \[
  \lim_{N \to \infty} \frac{1}{N} \sum_{k=0}^{N-1} F( Y^k ) =
  \wh{\expE}[ F(Y)]
  \] 
  holds $\Pm$-almost surely.
\end{theo}

In particular, the limit $ \wh{\expE}[F(Y)]$ is independent of
$X_0$. Using Theorem~\ref{theo:empirical}, several interesting
statistics of the transition paths can be expressed in terms of the
quantities we have defined. Actually, Proposition~\ref{prop:exitdistr}
is an immediate corollary of Theorem~\ref{theo:empirical}, by choosing
$F(Y^k) = f(Y^k_0)$, so we will not give a separate proof of
Proposition~\ref{prop:exitdistr}.

\subsection{Reaction rate}

Let $N_T$ be the number of $A \to B$ reactive trajectories up to time
$T$: 
\begin{equation*}
  N_T = 1 + \max_k \{ k \geq 0 \mid \tau_{B, k}^+ \leq T \}. 
\end{equation*}
The \textbf{reaction rate} $\nu_R$ is defined by the limit
\begin{equation}\label{reacrate}
\nu_R = \lim_{T \to \infty} \frac{N_T}{T} = \lim_{k \to \infty} \frac{k}{\tau_{B, k}^+},
\end{equation}
and it is the rate of the transition from $A$ to $B$.  Also,
the limits
\begin{equation}\label{eq:TAB}
T_{AB} := \lim_{N \to \infty} \frac{1}{N} \sum_{k=0}^{N-1} (\tau_{B,k}^+ -
\tau_{A,k}^+)
\end{equation}
and
\begin{equation}\label{eq:TBA}
T_{BA} := \lim_{N \to \infty} \frac{1}{N} \sum_{k=0}^{N-1}
(\tau_{A,k+1}^+ - \tau_{B,k}^+)
\end{equation}
are the \textbf{expected reaction times} from $A \to B$ and $B \to A$,
respectively. The reaction rate from $A \to B$ and $B\to A$ are
  then given by $k_{AB} = T_{AB}^{-1}$ and $k_{BA} = T_{BA}^{-1}$.
Another interesting quantity is the {\bf expected crossover time} from
$A \to B$
\begin{equation}\label{eq:CAB}
C_{AB} := \lim_{N \to \infty} \frac{1}{N} \sum_{k=0}^{N-1} (\tau_{B,k}^+ -
\tau_{A,k}^-),
\end{equation}
which is the typical duration of the $A \to B$ reactive
intervals. Observe that $C_{AB} < T_{AB}$. Similarly, we define
\begin{equation}\label{eq:CBA}
C_{BA} := \lim_{N \to \infty} \frac{1}{N} \sum_{k=0}^{N-1} (\tau_{A,k+1}^+ -
\tau_{B,k}^-).
\end{equation}
The next result identifies these limits in terms of the committor
functions and the reactive exit and entrance distributions.

\begin{prop} \label{prop:reactrate} The limits \eqref{reacrate},
  \eqref{eq:TAB}, \eqref{eq:TBA}, \eqref{eq:CAB}, and \eqref{eq:CBA}
  hold $\Pm$-almost surely, and
  \begin{align*}
    & \nu_R = \nu = \int_{\RR^d} \rho(x) \nabla q(x) \cdot a(x) \nabla
    q(x) \ud x.
    % - \int_{\partial A} \rho(x) \wh{n}(x) \cdot
    % a(x)\nabla q(x) \ud \sigma_A(x).
    \\
    & T_{AB} = \int_{\partial A} \eta_A^+(\rd x) u_B(x) =
    \frac{1}{\nu_R} \int_{\RR^d} \rho(x) \wt{q}(x) \ud x. \\
    & T_{BA} = \int_{\partial B} \eta_B^+(\rd x) u_A(x) =
    \frac{1}{\nu_R} \int_{\RR^d} \rho(x) ( 1- \wt{q}(x)) \ud x.\\
    & C_{AB} = \int_{\partial A} \eta_A^-(\rd x) v_B(x) =
    \frac{1}{\nu_R} \int_{\RR^d} \rho(x)q(x)\wt{q}(x) \ud x. \\
    & C_{BA} = \int_{\partial B} \eta_B^-(\rd x) v_A(x) =
    \frac{1}{\nu_R} \int_{\RR^d} \rho(x)(1 - q(x))(1 - \wt q(x))\ud x.
  \end{align*}
  Here $u_B(x) = \expE[\tau^X_{B} \;|\; X_0 = x ]$ is the mean first
  hitting time of $X_t$ to $\wb{B}$, and $v_B(x) = \wh{\expE}[
  \tau^Y_B \;|\; Y_0 = x]$ is the mean first hitting time of $Y_t$ to
  $\wb{B}$. Similarly, if $q$ is replaced by $(1 - q)$ in the
  definition of $Y$, then $v_{A}(x) = \wh{\expE}[ \tau^Y_A \;|\; Y_0 =
  x]$. Recall that $\nu$ is the normalizing factor for the reactive
  exit and entrance distributions.
\end{prop}

The formula for $\nu_R$, $T_{AB}$, and $T_{BA}$ were obtained in
  \cite{EVa:06}. We also note that the crossover time for the transition path process
in one dimension was recently studied by \cite{Cerou:12}. 

\subsection{Density of transition paths}

We now consider the distribution $\rho_R$ as defined in \cite{EVa:06}:
\begin{equation}\label{rhoRdef}
  \rho_R(z) = \lim_{T \to \infty} \frac{1}{T} \int_0^T \delta(z - X_t) 
  \mathbb{I}_{R}(t) \ud t, \quad z \in \Theta,
\end{equation}
where $R$ is the random set of times at which $X_t$ is reactive: 
\[
R = \bigcup_{k = 0}^\infty [\tau_{A,k}^-,\tau_{B,k}^+].
\]
This distribution on $\Theta$ can be viewed as the density of
transition paths. By Proposition \ref{prop:reactrate}, and
Theorem~\ref{theo:empirical}, we can describe $\rho_R$ in terms of the
transition density for $Y_t$. Specifically, for any continuous and
bounded function $f:\Rm^d \to \Rm$, we have
\begin{equation*}
  \begin{aligned}
    \int_{\Theta} f(z) \rho_R(z) \ud z & = \nu_R \lim_{T \to \infty}
    \frac{1}{N_T}
    \int_0^T f(X_t) \mathbb{I}_{R}(t) \ud t \\
    & = \nu_R \lim_{N \to \infty} \frac{1}{N} \sum_{k=0}^{N-1}
    \int_0^{\tau_{B, k}^+-\tau_{A, k}^-} f\bigl(Y^k_t\bigr) \ud t \\
    & = \nu_R \, \wh{\expE}\left[ \int_0^{t_B} f(Y_t) \ud t \mid Y_0
      \sim \eta_A^- \right] \\
    & = \nu_R \int_0^\infty \int_{\Theta} Q_R(t,\eta_A^-, z) f(z)\ud
    z\, \ud t.
  \end{aligned}
\end{equation*}
Here $Q_R(t,\eta_A^-,z)$ is the density of $Y_t$, with $Y_0 \sim
\eta_A^-$, and killed at $\partial B$
\begin{equation}
  Q_R(t, \eta_A^-, z)  = \mathbb{Q}(Y_t \in \rd z, \, t < t_{B} 
  \mid Y_0 \sim \eta_A^-), 
\end{equation}
and $t_{B}$ is the first hitting time of $Y_t$ to $\wb{B}$.  Hence,
for $z \in \Theta$,
\begin{equation}
  \begin{aligned}
    \rho_R(z) & = \nu_R \int_0^\infty Q_R(t,\eta_A^-, z)  \ud t. \label{rhoRPR}
  \end{aligned}
\end{equation}

\begin{prop}\label{prop:rhoR}
  For all $z \in \Theta$,
  \begin{equation}\label{eq:rhoR}
    \rho_R(z) = \rho(z) q(z) \wt{q}(z). 
  \end{equation}
\end{prop}

This formula for $\rho_R$ was first derived in \cites{Hummer:04,
    EVa:06}.

\subsection{Current of transition paths}
The density $Q_R(t,\eta_A^-,z)$ satisfies the adjoint equation
\[
\frac{\partial}{\partial t} Q_R(t,\eta_A^-,z) = (L^q)^{\ast}
Q_R(t,\eta_A^-,z), \quad z \in \Theta
\]
where $(L^q)^{\ast}$ is the adjoint of $L^q$:
\[
(L^q)^{\ast} u = \sum_{i,j} (a_{ij}(z) u(z))_{z_i z_j} - \sum_{i} (K_i(z) u(z))_{z_i}
\]
and $K$ is defined by (\ref{Kvecdef}). Integrating from $t = 0$ to $t
= \infty$ we see that $\rho_R(z)$ satisfies
\[
(L^q)^{\ast} \rho_R(z) = 0, \quad z \in \Theta.
\]
In divergence form, this equation is
\begin{equation}\label{eq:divJR}
\nabla_z \cdot J_R(z) = 0,
\end{equation}
where the vector field
\begin{equation}\label{eq:JR}
  \begin{aligned}
    J_R(z) & = \rho_R(z) \biggl(b(z) - \frac{2 a
      \nabla q(z)}{q(z)}\biggr) + \divop( a(z)\rho_R(z)) \\
    & = \Bigl(b(z) \rho(z) - \divop\bigl(a(z) \rho(z)\bigr)\Bigr) q(z)
    \wt{q}(z) + \rho(z) a(z) \Bigl( \wt{q}(z) \nabla q(z) - q(z)
    \nabla \wt{q}(z) \Bigr).
  \end{aligned}
\end{equation}
is continuous over $\wb{\Theta}$. The vector field $J_R(z)$,
identified in \cite{EVa:06}, may be regarded as the \textbf{current of
  transition paths} (see Remark~\ref{remark:current}).  Observe that if the SDE 
\eqref{XSDE} is reversible, we have $\wt{q} = 1 - q$ and
\begin{equation*}
  b(z) \rho(z) - \divop(a(z) \rho(z)) = 0,  
\end{equation*}
and hence the current given by (\ref{eq:JR}) simplifies to 
\begin{equation*}
J_R(z) = \rho(z) a(z) \nabla q(z).
\end{equation*}
This was observed already in \cite{EVa:06}.

\smallskip

On the boundary, the current \eqref{eq:JR} is related to the reactive
exit and entrance distributions.
\begin{prop}\label{prop:Jboundary}
  We have
  \begin{equation*}
    J_R = \rho a \nabla q \; 
    \text{on }\partial A, \quad \text{and} \quad 
    J_R = - \rho a \nabla \wt{q}, \; \text{on }\partial B,
  \end{equation*}
  and hence,
  \begin{equation*}
    \eta_A^-(\rd x) = - \nu_R^{-1} \wh{n}(x) \cdot J_R(x) \ud \sigma_A(x)
    \quad \text{and} \quad 
    \eta_B^+(\rd x) = \nu_R^{-1} \wh{n}(x) \cdot J_R(x) \ud \sigma_B(x).
  \end{equation*}
\end{prop}

As an immediate corollary, we have an additional formula for the
reaction rate.
\begin{cor}\label{cor:SJR}
  Let $S$ be a set with smooth boundary that contains $A$ and
  separates $A$ and $B$, we have
  \begin{equation}
    \nu_R = \int_{\partial S} \wh{n}(x) \cdot J_R(x) \ud \sigma_{S}(x),
  \end{equation}
  where $\wh{n}$ is the unit normal vector exterior to $S$. 
\end{cor}

The current $J_R$ generates a (deterministic) flow in $\wb{\Theta}$
stopped at $\partial B$:
\begin{equation}\label{eq:Z}
  \frac{\ud Z_t^z}{\ud t} = J_R(Z_t^z), \quad \text{for } 0 \leq t \leq t_B, 
  \quad Z_0^z = z
\end{equation}
where $t_B = t_B(z)$ is the time at which $Z_t$ reaches $\partial B$. As $J_R$ is
divergence free in $\Theta$, $J_R \cdot \wh{n} < 0$ on $\partial A$,
and $J_R \cdot \wh{n} > 0$ on $\partial B$, $t_B(z)$ is finite for any
$z \in \wb{\Theta}$. The flow naturally defines a map
$\Phi_{J_R}: \partial A \to \partial B$: given any point $z
\in \partial A$, we define
\begin{equation}\label{eq:phiJ}
  \Phi_{J_R}(z) = Z_{t_B}^z \in \partial B. 
\end{equation}

\begin{prop}\label{prop:JR}
For any $f \in C^1(\Rm^d)$,
\begin{equation}\label{eq:defJR}
  \int_{\partial B} f(x) \eta_B^+(\rd x) 
  - \int_{\partial A} f(x) \eta_A^-(\rd x) 
  = \frac{1}{\nu_R} \int_{\Theta} J_R \cdot \nabla f \ud x.
\end{equation}
In particular,
\begin{equation*}
  \Phi_{J_R, \ast}(\eta_A^-) = \eta_B^+, 
\end{equation*}
where $\Phi_{J_R, \ast}(\eta_A^-)$ is the pushforward of the measure $\eta_A^-$
by the map $\Phi_{J_R}$.
\end{prop}

Hence, $J_R$ characterizes ``the flow of reactive trajectories'' from
$A$ to $B$.

\begin{rmk}\label{remark:current}
  Note that by Proposition~\ref{prop:exitdistr} and
  Proposition~\ref{prop:etaB+}, the left hand side of \eqref{eq:defJR}
  is equal, $\Pm$-almost surely, to the limit
  \begin{equation*}
    \lim_{N \to \infty} \frac{1}{N} \sum_{n=0}^{N-1} \bigl( f(X_{\tau_{B, n}^+})
    - f(X_{\tau_{A, n}^-}) \bigr).
  \end{equation*}
  If $X_t$ was differentiable, we would have
  \begin{equation*}
    \begin{aligned}
      \lim_{N \to \infty} \frac{1}{N} \sum_{n=0}^{N-1} \Bigl(
      f(X_{\tau_{B, n}^+}) - f(X_{\tau_{A, n}^-}) \Bigr) & = \lim_{T
        \to \infty} \frac{1}{\nu_R} \frac{1}{T} \int_0^{T}
      1_R(t) \frac{\rd}{\rd t} f(X_t) \ud t \\
      & \text{``} = \frac{1}{\nu_R} \int_{\Theta} \ud x \nabla f(x)
      \cdot \lim_{T \to \infty} \frac{1}{T} \int_0^T \dot{X}_t
      \delta(x - X_t) 1_R(t) \ud t \ \text{''},
    \end{aligned}
  \end{equation*}
  Combining this with Proposition~\ref{prop:JR}, we arrive at a
  formal characterization of $J_R$
  \begin{equation*}
    J_R    \text{``} = \lim_{T \to \infty}
    \frac{1}{T} \int_0^T \dot{X}_t \delta(x - X_t) 1_R(t) 
    \ud t\ \text{''}. 
  \end{equation*}
  This formal expression was used in \cite{EVa:06} to define $J_R$.
\end{rmk}

\subsection{Related work} \label{sec:related}

As we have mentioned, our work is closely related to the transition path theory developed by E and
  Vanden-Eijnden \cites{EVa:06, MeScVa:06, EVa:10}, which is
  a framework for studying the transition paths. In particular, based on
  the committor function, formula for reaction rate, density and
  current of transition paths were obtained in \cite{EVa:06}. Our main motivation is to understand the probability law of the
transition paths. The main results Theorem~\ref{theo:uniqueYAB},
Theorem~\ref{theo:XtauAmin}, and Theorem~\ref{theo:empirical} identify
an SDE which characterizes the law of the transition paths in $C([0,\infty))$. Therefore,
as an application of these results, we are able to give rigorous
proofs for the formula for reaction rate, density and current of
transition paths in \cite{EVa:06}. We note that in the discrete case,
a generator analogous to \eqref{Lconditioned} was also proposed very
recently in~\cite{EricPreprint} for Markov jumping processes.

The transition paths start at $\partial A$ and terminate at $\partial B$,
and hence they can be viewed as paths of a bridge process between $\wb{A}$
and $\wb{B}$. In this perspective, our work is related to the conditional
path sampling for SDEs studied in \cites{StuartVossWiberg:04,
  ReznikoffVandenEijnden:05, HairerStuartVossWiberg:05,
  HairerStuartVoss:07}. In those works, stochastic partial
differential equations were proposed to sample SDE paths with fixed
end points. However, the paths considered were different from the
transition paths as their time duration is fixed a priori. It
would be interesting to explore SPDE-based sampling strategies for the
transition path process identified in Theorem~\ref{theo:uniqueYAB}.

Let us also point out that in the work we present here we do not assume that the noise $\sigma$ is small, as is the case in the asymptotic results of \cite{BoEcGaKl:04, BoGaKl:05, Cerou:12}, which we have mentioned already, and also in some other works, such as the large deviation theory of Freidlin and Wentzell \cite{FW:84}.

\bigskip

The rest of the paper is organized as follows. Theorem
\ref{theo:uniqueYAB} and Theorem \ref{theo:XtauAmin} are proved in
Section \ref{sec:TPP}. In Section \ref{sec:REED} we prove Lemma
\ref{lem:sameMass}, Proposition \ref{prop:etaB+} and Theorem
\ref{theo:empirical} related to the reactive entrance and exit
distributions. As we have mentioned, Proposition \ref{prop:exitdistr}
follows immediately from Theorem \ref{theo:empirical}, so we do not
give a separate proof of it. Proposition \ref{prop:reactrate},
Proposition \ref{prop:rhoR}, Proposition \ref{prop:Jboundary},
Corollary \ref{cor:SJR}, and Proposition \ref{prop:JR} are proved in
Section \ref{sec:appltpt}.

\section{The Transition Path Process} \label{sec:TPP}

\begin{proof}[Proof of Theorem \ref{theo:uniqueYAB}]

  Without loss of generality, we prove the theorem in the case that
  $\xi \equiv y_0$ is a single point in $\wb{\Theta}$. The interesting
  aspect of the theorem is that $y_0$ is allowed to be on $\partial
  \Theta$, since the drift term is singular at $\partial \Theta$. If
  we assume that $y_0 \in \Theta$, then existence of a unique strong
  solution up to the time $\tau_A \wedge \tau_B$ follows from standard
  arguments, since $K(y)$ is Lipschitz continuous in the interior of
  $\Theta$.  That is, if $y_0 \in \Theta$, there is a unique,
  continuous $\wh{\mathcal{F}}_t$-adapted process $Y_t$ which
  satisfies
  \begin{equation}\label{YSDE1}
    Y_t = y_0 + \int_0^{t \wedge (\tau_A \wedge \tau_B)}
    K(Y_s) \ud s + \int_0^{t \wedge (\tau_A \wedge \tau_B)} \sqrt{2}\,
    \sigma(Y_s) \ud\wh{W}_s, \quad t \geq 0.  
  \end{equation}
  Moreover, if $y_0 \in \Theta$, then we must have $\tau_A > \tau_B >
  0$ almost surely. This follows from an argument similar to the proof
  of \cite{KaSh:91}*{Proposition 3.3.22, p.~161}. Specifically, we
  consider the process $z_t = 1/q(Y_t) \in \Rm$, which satisfies
  \[
  z_{t \wedge \tau} = z_0 - \int_0^{t \wedge \tau} \sqrt{2} (z_s)^2
  \nabla q \cdot \sigma \ud\wh{W}_s
  \]
  where $\tau = \tau_B \wedge \tau_\epsilon$ with $\tau_\epsilon =
  \inf \{ t > 0 \mid q(Y_t) = \epsilon \}$.  Since $\tau < \infty$
  with probability one, we have
  \[
  z_0 = \wh{\expE}[z_{t \wedge \tau}] = \frac{1}{q(\epsilon)}
  \mathbb{Q}( \tau_\epsilon < \tau_B) + \mathbb{Q}(\tau_\epsilon >
  \tau_B).
  \]
  Hence $\mathbb{Q}( \tau_\epsilon < \tau_B) \leq q(\epsilon)(z_0-1)$.
  So, $\mathbb{Q}(\tau_A < \tau_B) \leq \lim_{\epsilon \to 0}
  \mathbb{Q}( \tau_\epsilon < \tau_B)=0$.

  \bigskip

  Now suppose $y_0 \in \partial A$. In consideration of the comments
  above, it suffices to prove the desired result with $\tau_B$
  replaced by $\tau_r$, the first hitting time to $\partial B_r(y_0)
  \cap \Theta$, where $B_r(y_0)$ is a ball of radius $r > 0$ centered
  at $y_0$. Thus, we want to prove existence and pathwise uniqueness
  of a continuous $\wh{\mathcal{F}}_t$-adapted process $Y_t:[0,\infty)
  \to \bar \Theta$ satisfying
  \begin{equation}\label{Yteqn} 
    Y_t = y_0 + \int_0^{t \wedge \tau_r} K(Y_s) \ud s 
    + \int_0^{t \wedge \tau_r} \sqrt{2}\, \sigma(Y_s) \ud\wh{W}_s, 
  \end{equation}
  where
  \[
  \tau_r = \inf \left \{ t \geq 0 \mid Y_t \in \partial B_r(y_0) \cap
    \Theta \, \right \}.
  \]
  It will be very useful to define a new coordinate system in the set
  $B_r^+(y_0) = B_r(y_0) \cap \Theta$ and to consider the problem in
  these new coordinates. For $r > 0$ small enough we can define a
  $C^3$ map $(h^{(1)}(y),\dots,h^{(d-1)}(y),q(y)): \overline{B_r^+(y_0)} \to
  \Rm^{d-1} \times [0,\infty)$, such that the scalar functions
  $h^{(i)}(y): \overline{B_r^+(y_0)} \to \Rm$ satisfy
  \begin{equation}\label{hqOrthog} 
    \avg{\nabla h^{(i)}(y), a(y) \nabla q(y)} = 0, \quad \forall\;y \in
    \overline{B_r^+(y_0)}, \qquad i = 1,\dots,d-1. 
  \end{equation}
  Furthermore, the map may be constructed so that it is invertible on
  its range and that the inverse is $C^3$. The existence of such a map
  follows from the regularity of $\partial A$, the regularity of $q$,
  and the fact that $\avg{\wh{n}, a \nabla q} \neq 0$ on $\partial A$
  by \eqref{hopf}.

  For two initial points $x_1, x_2 \in \Theta$, let $Y^{x_1}_t$ and
  $Y^{x_2}_t$ denote the unique solutions to \eqref{YSDE1} with
  $Y^{x_1}_0 = x_1$ and $Y^{x_2}_0 = x_2$ respectively. That is, 
  \begin{equation}\label{YxSDE}
    Y^{x}_t = x + \int_0^{t \wedge \tau^x_{B}}  K(Y^x_s) \ud s 
    + \int_0^{t \wedge \tau^x_{B}}
    \sqrt{2}\, \sigma(Y^x_s)\ud\wh{W}_s, \quad t \geq 0,
  \end{equation}
  where $\tau_B^x$ is the first hitting time of $Y^x_t$ to $\partial
  B$.  Changing to the coordinate system defined by
  $(h^{(1)}(y),\dots,h^{(d-1)}(y),q(y))$, we denote
  \[
  (h_{1,t},q_{1,t}) = (h(Y^{x_1}_t), q(Y^{x_1}_t)) \quad \text{and}
  \quad (h_{2,t},q_{2,t}) = (h(Y^{x_2}_t), q(Y^{x_2}_t)).
  \]
  Let $\tau_r^{1}$ and $\tau_r^{2}$ denote the first hitting times of
  $Y^{x_1}_t$ and $Y^{x_2}_t$ to the set $\partial B_r(y_0) \cap
  \Theta$. The processes $(h_{1,t},q_{1,t})$ and $(h_{2,t},q_{2,t})$
  are well-defined up to the times $\tau^1_r$ and $\tau^2_r$,
  respectively.

  \medskip

  We can control the difference between $(h_{1,t},q_{1,t})$ and
  $(h_{2,t},q_{2,t})$:
  %%%%%%%%%%%%%%%%%%%%%%%%%%%%%% 
  \begin{lem}\label{lem:qhx1x2}
    There is a constant $C$ such that for all $x_1, x_2 \in B_{r/2}(y_0)
    \cap \Theta$
    \begin{align*}
      & \wh{\expE}\left[\max_{t \in [0,T] }(q_{1,t\wedge \tau} -
        q_{2,t \wedge \tau})^2 \right] \leq C |x_1 - x_2|^{1/2},  \\
      \intertext{and} & \wh{\expE}\left[\max_{t \in [0,T] }
        \abs{h_{1,t\wedge \tau} - h_{2,t \wedge \tau}}^2 \right] \leq
      C |x_1 - x_2|,
    \end{align*}
    where $\tau = \tau^{1}_r \wedge \tau^{2}_r$.
  \end{lem}
  The proof of Lemma~\ref{lem:qhx1x2} will be postponed. One immediate
  corollary is the following.
  \begin{cor}\label{cor:chebyY12}
    There is a constant $C$ such that
    for all $x_1, x_2 \in B_{r/2}(y_0) \cap \Theta$ 
    \begin{equation}\label{chebyY12} 
      \mathbb{Q}\biggl( \max_{0 \leq t \leq (T \wedge \tau)} |Y^{x_1}_t
      - Y^{x_2}_t| > \alpha \biggr) \leq C \alpha^{-2} |x_{1} -
      x_{2}|^{1/2}, \quad \forall \; \alpha > 0, 
    \end{equation}
    where $\tau = \tau^{1}_r \wedge \tau^{2}_r$.
  \end{cor}

  \begin{proof}%[Proof of Corollary~\ref{cor:chebyY12}]
    On the closed set $\{ z \in \Rm^{d} \mid z = (h(y),q(y)), \ y \in
    \overline{B_r^+(y_0)} \}$, the map $y \mapsto (h(y),q(y))$ is
    invertible with a continuously differentiable inverse. Hence there
    is a constant $C$, depending only on the map $y \mapsto
    (h(y),q(y))$ such that
    \[
    |Y^{x_1}_t - Y^{x_2}_t| \leq C \left(|h_{1,t} -
      h_{2,t}| + |q_{1,t} - q_{2,t}|\right), \quad \forall\; t \in [0,\tau].
    \]
    By combining this bound with Chebychev's inequality and
    Lemma~\ref{lem:qhx1x2} we obtain \eqref{chebyY12}.
  \end{proof}

  \medskip 

  Now suppose $y_0 \in \partial A$. Let $\{x_n\}_{n=1}^\infty \subset
  \Theta$ be a given sequence such that $x_n \to y_0$ as $n \to
  \infty$. For each $n$, define $Y^{x_n}_t$ by \eqref{YxSDE}, and let
  $\tau^n_r$ denote the first hitting time of $Y^{x_n}_t$ to $\partial
  B_r(y_0) \cap \Theta$.  We may choose the points $x_n$ so that $|x_n
  - y_0| \leq 25^{-n}$. Define $\wh \tau^n = \tau^{n+1}_r \wedge
  \tau_r^{n}$. Applying Corollary \ref{cor:chebyY12}, we conclude 
  \begin{equation*}
    \mathbb{Q}\biggl( \max_{0 \leq t \leq (T \wedge \wh\tau^n)}
    |Y^{x_{n+1}}_t - Y^{x_n}_t| > 2^{-n} \biggr) \leq C 2^{2n} 5^{-n}.
  \end{equation*}
  Therefore, by the Borel-Cantelli lemma, the series 
  \begin{equation}
    \sum_{n
      =1}^\infty \max_{0 \leq t \leq (T \wedge \wh \tau^n)}
    |Y^{x_{n+1}}_t - Y^{x_n}_t| < \infty \label{seriesconv} 
  \end{equation}
  with probability one. Let us define
  \begin{equation}\label{taurhopos} 
    \tau_r = \liminf_{n \to
      \infty} \tau_r^{n} = \liminf_{n \to \infty} \wh
    \tau^n. 
  \end{equation}
  We will prove that $\tau_r$ is positive:
  
  \medskip
  
  \begin{lem}\label{lem:taupos} For all $r > 0$ sufficiently small,
    $\mathbb{Q}(\tau_r > 0 ) = 1$.
  \end{lem}

  \medskip

  In view of \eqref{seriesconv} and Lemma~\ref{lem:taupos}, we
  conclude that there must be a continuous process $Y_t$ such that,
  with probability one,
  \[
  Y^{x_n}_t \to Y_t
  \]
  uniformly on compact subsets of $[0,\tau_r)$, as $n \to \infty$. Let
  us define 
  \begin{equation} 
    \wb \tau_{r/2} = \inf \{ t \geq 0 \mid Y_t
    \in \partial B_{r/2}(y_0) \cap \Theta \}. \label{bartaur2def}
  \end{equation}

  \medskip
  \begin{lem}\label{taur2bound} 
    For all $r > 0$ sufficiently small, $\mathbb{Q}(\wb \tau_{r/2} \in
    (0,\tau_r)) = 1$, and $\wb \tau_{r/2}$ is stopping time with
    respect to $\wh{\mathcal{F}}_t$.
  \end{lem}
  \medskip

  We will postpone the proof of Lemma \ref{lem:taupos} and Lemma \ref{taur2bound}. Since $\wb \tau_{r/2} < \tau_r$, $Y^{x_n}_t \to Y_t$ uniformly on
  $[0,\wb \tau_{r/2}]$. Let us now replace $Y_t$ by the stopped
  process $Y_{t \wedge \wb \tau_{r/2}}$. Since each $Y^{x_n}_t$ is
  $\wh{\mathcal{F}}_t$-adapted, so is the limit $Y_t$. We claim that
  $Y_t$ satisfies 
  \begin{equation} 
    Y_t = y_0 + \int_0^{t \wedge \wb \tau_{r/2}} K(Y_s)
    \ud s + \int_0^{t \wedge \wb \tau_{r/2}} \sqrt{2}\,
    \sigma(Y_s)\ud\wh{W}_s, \quad t \geq 0. \label{ylimintegral}
  \end{equation}
  Since $Y^{x_n}_t \to Y_t$ uniformly on $[0,\wb \tau_{r/2}]$, we have
  $(q(Y^{x_n}_t), h(Y^{x_n}_t)) \to (q(Y_t),h(Y_t))$ uniformly on
  $[0,\wb \tau_{r/2}]$, and $(q_t,h_t) = (q(Y_t),h(Y_t))$ satisfies
  \begin{equation}
    h_t = h_0 + \int_0^{t \wedge \wb \tau_{r/2}} f(q_s,h_s)
    \ud s + \int_0^{t \wedge \wb \tau_{r/2}} m(q_s,h_s) 
    \ud\wh{W}_s , \label{hlimintegral} 
  \end{equation} 
  and
  \begin{equation}\label{qintlim}
    q_t - \int_0^{t \wedge \wb \tau_{r/2}} g(q_s,h_s) \cdot \ud \wh{W}_s 
    = \lim_{n \to \infty} \int_0^{t \wedge \tau^{n}_r} 
    \frac{|g(q_s^{x_n},h_s^{x_n})|^2}{q_s^{x_n}} \ud s. 
  \end{equation}
  for all $t \in [0,\wb \tau_{r/2}]$, where $(q_t^{x_n},h_t^{x_n}) =
  (q(Y^{x_n}_t), h(Y^{x_n}_t))$. (Recall $q_0 = 0$.) Since $q_s^{x_n}
  > 0$, the last limit can be bounded below using Fatou's lemma:
  \begin{equation}
    q_t  - \int_0^{t \wedge \wb \tau_{r/2}} g(q_s,h_s) \cdot \ud \wh{W}_s \geq 
    \int_0^{t \wedge \wb \tau_{r/2}} \liminf_{n \to \infty}  
    \frac{|g(q_s^{x_n},h_s^{x_n})|^2}{q_s^{x_n}} \ud s = \int_0^{t \wedge \wb \tau_{r/2}} 
    \frac{|g(q_s,h_s)|^2}{q_s} \ud s 
  \end{equation}
  Recall that $|g(q_s,h_2)|^2 \geq C_r > 0$. In particular, with probability one, the random set $H = \{ s \in
  [0,\wb \tau_{r/2}] \mid q_s = 0 \}$ must have zero Lebesgue measure;
  if that were not the case, then we would have
  \[
  - \int_0^{t \wedge\wb \tau_{r/2}} g(q_s,h_s) \cdot \ud \wh{W}_s =
  +\infty,
  \]
  for all $t$ in a set of positive Lebesgue measure, an event which
  happens with zero probability. Therefore, by Fubini's theorem,
  \[
  0 = \wh{\expE} \int_0^T \mathbb{I}_H(s) \ud s = \int_0^T \mathbb{Q}(
  s < \wb \tau_{r/2} \,,\; q_s = 0 ) \ud s
  \]
  which implies that $\mathbb{Q}( s < \wb \tau_{r/2}, \; q_s = 0 ) =
  0$ for almost every $s \geq 0$. Since $\wb \tau_{r/2} > 0$ almost
  surely, this implies that we may choose a deterministic sequence of
  times $t_n \in (0, 1/n]$ such that, almost surely, $q_{t_n} > 0$ for
  $n$ sufficiently large. By then applying the same argument as when
  $y_0 \in \Theta$, we conclude that $q_t > 0$ for all $t >
  t_n$. Hence, $q_t > 0$ for all $t > 0$ must hold with probability
  one.

  Since $q_t$ is continuous, we now know that for any $\epsilon > 0$,
  \[
  \min_{ t > \epsilon } q_t > 0.
  \]
  holds with probability one. In particular,
  \[
  \liminf_{n \to \infty} \min_{ t > \epsilon } q_t^{x_n} > 0,
  \]
  so that
  \[
  \lim_{n \to \infty} \int_\epsilon^{t \wedge \tau^{n}}
  \frac{|g(q_s^{x_n},h_s^{x_n})|^2}{q_s^{x_n}} \ud s =
  \int_\epsilon^{t \wedge \wb \tau_{r/2}} \frac{|g(q_s,h_s)|^2}{q_s}
  \ud s,
  \]
  almost surely. Since $q_t$ is continuous at $t = 0$, we also know
  that
  \[
  \lim_{\epsilon \to 0} \lim_{n \to \infty} \int_0^{t \wedge \tau^{n}
    \wedge \epsilon} \frac{|g(q_s^{x_n},h_s^{x_n})|^2}{q_s^{x_n}} \ud
  s = \lim_{\epsilon \to 0} \left( q_\epsilon - \int_0^{t \wedge \wb
      \tau_{r/2} \wedge \epsilon} g(q_s,h_s) \cdot \rd\wh{W}_s \right)
  = 0
  \]
  almost surely. Returning to \eqref{qintlim} we now conclude that
  \begin{equation}\label{qlimintegra}
    \begin{aligned}
      q_t - \int_0^{t \wedge \wb \tau_{r/2}} g(q_s,h_s) \cdot
      d\wh{W}_s & = \lim_{\epsilon \to 0} \lim_{n \to \infty}
      \int_0^{t \wedge \tau^{n} \wedge \epsilon }
      \frac{|g(q_s^{x_n},h_s^{x_n})|^2}{q_s^{x_n}} \ud s  \\
      & \quad + \lim_{\epsilon \to 0} \lim_{n \to \infty}
      \int_\epsilon^{t \wedge \tau^{n}}
      \frac{|g(q_s^{x_n},h_s^{x_n})|^2}{q_s^{x_n}}\ud s \\
      & = \lim_{\epsilon \to 0} \int_\epsilon^{t \wedge \wb
        \tau_{r/2}} \frac{|g(q_s,h_s)|^2}{q_s} \ud s \\
      & = \int_0^{t \wedge \wb \tau_{r/2}} \frac{|g(q_s,h_s)|^2}{q_s}
      \ud s
    \end{aligned}
  \end{equation}
  holds with probability one. Equation \eqref{ylimintegral} for $Y_t$ now
  follows from \eqref{hlimintegral} and \eqref{qlimintegra} by
  changing coordinates.

  Except for the proofs of Lemma~\ref{lem:qhx1x2},
  Lemma~\ref{lem:taupos}, and Lemma~\ref{taur2bound}, we have now
  established existence of a strong solution $Y_t$ to \eqref{Yteqn}
  (with $r$ replaced by $r/2$).  The uniqueness of the solution
  follows by the same arguments. Suppose that $Y^{1}_t$ and $Y^{2}_t$
  both solve \eqref{Yteqn} with the same Brownian motion and the same
  initial point $Y^1_0 = Y^2_0 = y_0$. Then
  Corollary~\ref{cor:chebyY12} implies that, $\mathbb{Q}$ almost
  surely, $Y^{1}_t = Y^{2}_t$ for all $t \in [0,\tau^1_r \wedge
  \tau^2_r]$ where $\tau^1_r$ and $\tau^2_r$ are the corresponding
  hitting times to $\partial B_r(y_0) \cap \Theta$. In particular,
  $\tau^1_r = \tau^2_r$.  This proves pathwise uniqueness. 
\end{proof}

We now prove Lemma~\ref{lem:qhx1x2}, Lemma~\ref{lem:taupos} and
Lemma~\ref{taur2bound} to complete the proof of
Theorem~\ref{theo:uniqueYAB}.

  \begin{proof}[Proof of Lemma~\ref{lem:qhx1x2}]
    By It\^o's formula the process $(h_1,q_1) = (h_{1,t},q_{1,t})$
    satisfies
    \begin{align}
      \ud h_1 & = f(q_1,h_1) \ud t + m(q_1,h_1) \ud\wh{W}_t,
      \label{newcoord1} \\
      \ud q_1 & = \frac{|g(q_1,h_1)|^2}{q_1} \ud t + g(q_1,h_1) \cdot
      \rd\wh{W}_t, \label{newcoord1b}
    \end{align}
    for $0 \leq t \leq \tau^{1}_r$, where the functions $g = \sqrt{2}
    (\nabla q)^{\TT} \sigma \in \Rm^{d}$, $f = L h \in \Rm^{d-1}$, and $m
    = \sqrt{2}(\nabla h)^{\TT} \sigma \in \Rm^{(d-1) \times d}$, are all
    Lipschitz continuous in their arguments over
    $\wb{B_r^+}$. Similarly, $(h_2,q_2) = (h_{2,t},q_{2,t})$ satisfies
    \begin{align}
      \ud h_2 & =  f(q_2,h_2) \ud t + m(q_2,h_2) \ud\wh{W}_t \label{newcoord2} \\
      \ud q_2 & = \frac{|g(q_2,h_2)|^2}{q_2} \ud t + g(q_2,h_2) \cdot
      \rd\wh{W}_t, \label{newcoord2b} 
    \end{align}
    for $0 \leq t \leq \tau^{2}_r$. Notice that the choice of
    coordinates satisfying \eqref{hqOrthog} has eliminated a
    potentially singular drift term in the equations for $h_{1,t}$ and
    $h_{2,t}$. On the other hand, the drift term in the equations for
    $q_1$ and $q_2$ blows up near the boundary $q = 0$. Indeed, if $r
    > 0$ is small enough, by \eqref{hopf} there is a constant $C_r >
    0$ such that
    \begin{equation}
      \inf_{y \in \wb{B^+_r}} 2 \avg{\nabla q(y)), a(y) \nabla q(y)} \geq 
      2 \lambda \inf_{y \in \wb{B^+_r}} \abs{\nabla q(y)} \geq C_r. 
    \end{equation}
    Hence,
    \begin{equation}\label{hopfglower}
      \begin{aligned}
        |g(q_{1,t}, h_{1,t})|^2 & = 2 \avg{\nabla q(Y^{x_1}_t), a(Y^{x_1}_t) \nabla
          q(Y^{x_1}_t)} \geq 2 \lambda \inf_{y \in \wb{B^+_r}}
        \abs{\nabla q(y)} \geq C_r > 0.
      \end{aligned}
    \end{equation}

    Letting $\tau = \tau_r^1 \wedge \tau_r^2$ and using
    \eqref{newcoord1} and \eqref{newcoord2}, we compute
    \begin{align*}
      \rd\, |h_1 - h_2|^2 & = 2 (h_1 - h_2)^{\TT} (f(q_1,h_1) -
      f(q_2,h_2)) \ud t  \\
      & \quad + 2 (h_1 - h_2)^{\TT} (m(q_1,h_1) - m(q_2,h_2)) \ud\wh{W}_t \\
      & \quad + \tr \bigl( (m(q_1,h_1) - m(q_2,h_2))(m(q_1,h_1) -
      m(q_2,h_2))^{\TT}\bigr) \ud t
    \end{align*} 
    for $0 \leq t \leq \tau$. In particular, 
    \begin{equation}\label{hbound12delta}
      \begin{aligned}
        \wh{\expE}\, \bigl[ |h_{1,t\wedge \tau} - h_{2,t \wedge
          \tau}|^2\bigr] & \leq C \int_{0}^{t} \wh{\expE}\, \bigl[
        \mathbb{I}_{[0,\tau]}(s)
        (q_{1,s} - q_{2,s})^2\bigr] \ud s \\
        & \quad + C \int_{0}^{t} \wh{\expE}\,\bigl[
        \mathbb{I}_{[0,\tau]}(s)|h_{1,s} -
        h_{2,s}|^2\bigr] \ud s + C|x_1 - x_2|, \\
        & \leq C \int_{0}^{t} \wh{\expE}\,\bigl[ (q_{1,s\wedge \tau} -
        q_{2,s\wedge \tau})^2\bigr]\ud s \\
        & \quad + C \int_{0}^{t} \wh{\expE}\,\bigl[ |h_{1,s\wedge
          \tau} - h_{2,s\wedge \tau}|^2\bigr]\ud s + C|x_1 - x_2|,
      \end{aligned}
    \end{equation}
    holds for all $t \geq 0$.

    From \eqref{newcoord1b} and \eqref{newcoord2b} we also compute 
    \begin{equation}\label{dq2dif1}
      \begin{aligned}
        \rd\, (q_1 - q_2)^2 & = 2 (q_1 - q_2) \rd
        (q_1 - q_2) + |g_1 - g_2|^2 \ud t  \\
        & = 2 (q_1 - q_2) \left( \frac{|g_1|^2}{q_1} -
          \frac{|g_2|^2}{q_2}\right) \rd t \\
        & \quad + 2 (q_1 - q_2)(g_1 - g_2) \cdot \rd\wh{W}_t + |g_1 -
        g_2|^2 \ud t
      \end{aligned}
    \end{equation}
    for $0 \leq t \leq \tau$, where we have used the notation $g_1 =
    g(q_1,h_1)$ and $g_2 = g(q_2,h_2)$. We claim that there is a
    constant $C$, depending only on $r$, such that
    \begin{equation}\label{qdifObs} 
      2 (q_1 - q_2)
      \left( \frac{|g_1|^2}{q_1} - \frac{|g_2|^2}{q_2}\right) \leq C (|q_1
      - q_2|^2 + |h_1 - h_2|^2) 
    \end{equation}
    holds for all $t \leq \tau$, with probability one. Both sides of
    \eqref{qdifObs} are invariant when $(q_1,h_1)$ and $(q_2,h_2)$ are
    interchanged. So, we may assume $q_1 \leq q_2$ without loss of
    generality. We consider the following two possibilities. First,
    suppose that
    \begin{equation}\label{poss1}
      0 \leq q_1 \bigl\lvert\abs{g_1}^2 - \abs{g_2}^2\bigr\rvert 
      \leq (q_2 - q_1) \abs{g_1}^2. 
    \end{equation}
    Using this and $q_1 \leq q_2$ we have
    \begin{equation}\label{eq:qdifest1}
      \begin{aligned}
        2 (q_1 - q_2) \left( \frac{\abs{g_1}^2}{q_1} -
          \frac{|g_2|^2}{q_2}\right) & = 2 \frac{(q_1 - q_2)}{q_1 q_2}
        \left( q_2 |g_1|^2 - q_1 |g_2|^2 \right) \\
        & = 2 \frac{(q_1 - q_2)}{q_1 q_2} \left( (q_2 - q_1)|g_1|^2 -
          q_1(|g_2|^2 - |g_1|^2) \right) \\
        & \stackrel{\eqref{poss1}}{\leq} 0.
      \end{aligned}
    \end{equation}
    The other possibility is
    \begin{equation}\label{poss2}
      0 \leq (q_2 - q_1) \abs{g_1}^2 \leq q_1 \bigl
      \lvert\abs{g_1}^2 - \abs{g_2}^2\bigr\rvert.
    \end{equation}
    In this case, we have (also using $q_1 \leq q_2$)
    \begin{equation}\label{eq:qdifest2}
      \begin{aligned}
        2 (q_1 - q_2) \left( \frac{|g_1|^2}{q_1} -
          \frac{|g_2|^2}{q_2}\right) & = 2 \frac{(q_1 - q_2)}{q_1 q_2}
        \left( (q_2  - q_1)|g_1|^2 - q_1( |g_2|^2 - |g_1|^2) \right) \\
        & \leq - 2 \frac{(q_1 - q_2)}{q_1 q_2} q_1 ( |g_2|^2 - |g_1|^2) \\
        & \leq 2 \frac{\abs{q_1 - q_2}}{\abs{q_2}} \bigl\lvert
        \abs{g_2}^2
        - \abs{g_1}^2 \bigr\rvert \\
& \leq 2 \frac{\abs{q_1 - q_2}}{\abs{q_1}} \bigl\lvert
        \abs{g_2}^2
        - \abs{g_1}^2 \bigr\rvert \\
        & \stackrel{\eqref{poss2}}{\leq} 2 \frac{( |g_2|^2 -
          |g_1|^2)^2}{|g_1|^2}.
      \end{aligned}
    \end{equation}
    Therefore, since $\abs{g_1} \geq C_r > 0$ (by \ref{hopfglower}),
    we must have
    \[
    2 (q_1 - q_2) \left( \frac{|g_1|^2}{q_1} -
      \frac{|g_2|^2}{q_2}\right) \leq 2 C_r^{-2} ( |g_2|^2 -
          |g_1|^2)^2 \leq  C (|q_1 - q_2|^2 + |h_1 -
    h_2|^2).
    \]
    where $C > 0$ depends only on $r$. This establishes
    \eqref{qdifObs}.

    Returning to \eqref{dq2dif1} and controlling the first term on the
    right hand side of \eqref{dq2dif1} with \eqref{qdifObs}, we
    conclude that 
    \begin{equation}\label{q12delta}
      \begin{aligned} 
        \wh{\expE}\, \bigl[(q_{1,t\wedge \tau} - q_{2,t \wedge
          \tau})^2 \bigr] & \leq C \int_{0}^t
        \wh{\expE}\,\bigl[\mathbb{I}_{[0,\tau]}(s)
        (q_{1,s} - q_{2,s})^2\bigr] \ud s \\
        & \quad + C \int_{0}^t
        \wh{\expE}\,\bigl[\mathbb{I}_{[0,\tau]}(s)
        |h_{1,s} - h_{2,s}|^2\bigr] \ud s + C |x_1 - x_2|, \\
        & \leq C \int_{0}^t \wh{\expE}\,\bigl[(q_{1,s \wedge \tau} -
        q_{2,s\wedge \tau})^2 \bigr] \ud s \\
        & \quad + C \int_{0}^t \wh{\expE}\,\bigl[|h_{1,s\wedge \tau} -
        h_{2,s\wedge \tau}|^2\bigr] \ud s + C |x_1 - x_2|.
      \end{aligned}
    \end{equation}
    By combining \eqref{hbound12delta} and \eqref{q12delta} and
    applying Gronwall's inequality, we conclude that 
    \begin{equation}\label{gronwall1} 
      \wh{\expE}\, \bigl[
      |h_{1,t \wedge \tau} - h_{2,t\wedge \tau}|^2\bigr] +
      \wh{\expE}\, \bigl[(q_{1,t\wedge \tau} - q_{2,t\wedge \tau})^2\bigr] \leq C|x_1
      - x_2|\left(1 + t e^{C t}\right), \quad t \geq 0. 
    \end{equation}
    Using \eqref{dq2dif1} and \eqref{qdifObs} we also obtain 
    \begin{equation} \label{gronwall2}
      \begin{aligned}
        \wh{\expE}\, \Bigl[\max_{t \in [0,T] }(q_{1,t\wedge \tau} -
        q_{2,t \wedge \tau})^2\Bigr] & \leq
        C \int_{0}^T \wh{\expE}[(q_{1,s \wedge \tau} - q_{2,s\wedge \tau})^2] \ud s  \\
        & \quad + C \int_{0}^T \wh{\expE}[|h_{1,s\wedge \tau} -
        h_{2,s\wedge \tau}|^2] \ud s + C |x_1 - x_2| \\
        & \quad + \wh{\expE}\left[\max_{t \in [0,T] }
          V_{t} \right]
      \end{aligned}
    \end{equation}
    where $V_t$ is the martingale
    \[
    V_t = \int_0^{t\wedge \tau} 2 (q_1 - q_2)(g_1 - g_2) \cdot \ud \wh{W}_s.
    \]
    By the Burkholder-Davis-Gundy inequality (e.g. \cite{ReYo:04}*{Sec
      IV.4}) and \eqref{gronwall1}, we have
    \[
    \wh{\expE}\left[\max_{t \in [0,T] } V_t \right] \leq C
    \left( \int_0^T \wh{\expE}[ (q_{1,s\wedge \tau} - q_{2,s \wedge
        \tau})^2 ] \ud s \right)^{1/2} \leq C_T|x_1 - x_2|^{1/2}.
    \]
    This, together with \eqref{gronwall1} and \eqref{gronwall2}, gives
    us
    \[
    \wh{\expE}\left[\max_{t \in [0,T] }(q_{1,t\wedge \tau} - q_{2,t
        \wedge \tau})^2 \right] \leq C_T |x_1 - x_2|^{1/2}.
    \]
    Similar arguments for $h_1 - h_2$ lead to
    \[
    \wh{\expE}\left[\max_{t \in [0,T] }|h_{1,t\wedge \tau} - h_{2,t
        \wedge \tau}|^2 \right] \leq C_T |x_1 - x_2|.
    \]
  \end{proof}

\begin{proof}[Proof of Lemma \ref{lem:taupos}]
  Suppose $\tau_r = 0$ holds with probability $\epsilon > 0$. Because
  of \eqref{seriesconv} we may choose $m$ sufficiently large so that
  \[
  \sum_{n =m}^\infty \max_{0 \leq t \leq (T \wedge \wh \tau^n)}
  |Y^{x_{n+1}}_t - Y^{x_n}_t| < r/4
  \]
  holds with probability at least $1 - \epsilon/2$. Therefore, with
  probability at least $\epsilon/2$ we have both $\tau_r = 0$ and 
  \begin{equation}
    \liminf_{n \to \infty} |Y^{x_{n}}_{\tau_r^{n}} -
    Y^{x_m}_{\tau_r^{n}} | \leq r/4. \label{maxrhalf} 
  \end{equation}
  Recall that $|Y^{x_m}_0 - y_0| \leq 25^{-m}$. Let $m$ be larger, if
  necessary, so that $25^{-m} \leq r/4$. This and \eqref{maxrhalf}
  imply that
  \[
  \liminf_{n \to \infty} |Y^{x_{n}}_{\tau_r^{n}} - y_0| \leq
  \liminf_{n \to \infty} \left( |Y^{x_{n}}_{\tau_r^{n}} -
    Y^{x_m}_{\tau_r^{n}}| + | Y^{x_m}_{\tau_r^{_n}} - y_0|\right) \leq
  r/4 + 25^{-m} \leq r/2
  \]
  holds with probability at least $\epsilon/2$. However, this
  contradicts the fact that $Y^{x_n}_{\tau_r^{n}} \in \partial
  B_r(y_0)$ for all $n$. Hence, we must have $\tau_r > 0$ with
  probability one.
  % Since $|\wb Y_t - y_0| = O(\sqrt{t})$ for $t < T$, there is a
  % constant $C > 0$ such that $|Y^{x_m}_t - y_0| \leq C
  % \sqrt{\delta_m}$ holds for all $t \in [0,\delta_m]$ if $m$ is
  % large enough so that $\delta_m < T$.
\end{proof}

\medskip

\begin{proof}[Proof of Lemma \ref{taur2bound}]
  The fact that $\wb \tau_{r/2} > 0$ with probability one follows from
  an argument very similar to the proof of Lemma \ref{lem:taupos}. The
  fact that $\wb \tau_{r/2} < \tau_r$ will follow by showing that 
  \begin{equation}
    \limsup_{t \nearrow \tau_r} |Y_t - y_0| \geq
    r \label{limsupYttaur}
  \end{equation}
  holds with probability one.  First, suppose that $\tau^{n}_r <
  \tau_r$ and that
  \[
  \tau^{n}_r = \inf_{k \geq n} \tau^{k}_r
  \]
  Then by \eqref{seriesconv} we have
  \[
  |Y_{\tau^{n}_r} - y_0| \geq |Y^{x_n}_{\tau^{n}_r} - y_0| -
  |Y_{\tau^{n}_r} - Y^{x_n}_{\tau^{n}}| = r - |Y_{\tau^{n}_r} -
  Y^{x_n}_{\tau^{n}_r}| = r - R(n).
  \]
  where $R(n)$ is the series remainder
  \[
  R(n) = \sum_{k = n}^\infty \max_{0 \leq t \leq \tau^{n}_r}
  |Y^{x_{k+1}}_t - Y^{x_k}_t|
  \]
  which converges to zero, with probability one, as $n \to
  \infty$. So, with probability one, if there is an increasing
  sequence of such times $\tau^{{n_j}}_r \nearrow \tau_r$ as $j \to
  \infty$, we see that \eqref{limsupYttaur} must hold.  On the other
  hand, suppose there is no such sequence. Then we must have
  $\tau^{n}_r \geq \tau_r$ for $n$ sufficiently large. Hence
  $Y^{x_n}_t$ must converge to $Y_t$ uniformly on the closed interval
  $[0,\tau_r]$. Suppose $\tau^{n}_r \geq \tau_r$ and $\tau^{n}_r =
  \sup_{k \geq n} \tau^{k}_r$. Then for all $k \geq n$, we have
  \begin{equation*}
    \begin{aligned}
      |Y^{x_n}_{\tau^{k}_r} - y_0| & \geq |Y^{x_k}_{\tau^{k}_r} - y_0|
      - |Y^{x_n}_{\tau^{k}_r} -Y^{x_k}_{\tau^{k}_r}| \\
      & = r - |Y^{x_n}_{\tau^{k}_r} -Y^{x_k}_{\tau^{k}_r}| \geq r -
      M(n).
    \end{aligned}
  \end{equation*}
  Therefore, since $Y^{x_n}_t$ is continuous on $[0,\tau^{n}_r]$ and
  since $\tau_r = \liminf_{k \geq 0} \tau^{k}_r$, we have
  \[
  |Y^{x_n}_{\tau_r} - y_0| \geq r - M(n).
  \]
  Since $Y^{x_n}_{\tau_r} \to Y_{\tau_r}$ in this case and $Y_t$ is
  continuous on $[0,\tau_r]$, then with probability one, this case
  also implies that \eqref{limsupYttaur} holds.  Having established
  that $0 < \wb \tau_{r/2} < \tau_r$ we conclude that $Y^{x_n}_{t} \to
  Y_t$ uniformly on $[0,\wb \tau_{r/2}]$. Since each $Y^{x_n}_t$ is
  $\wh{\mathcal{F}}_t$-adapted, so is the limit $Y_t$. In particular,
  $\wb \tau_{r/2}$ is a stopping time.
\end{proof}

\medskip

\begin{rmk}
  Let us point out that if $y_0 \in \partial A$ and $T > 0$ is
  sufficiently small, the equation
  \begin{equation}\label{YbarODE}
    \wb Y(t) = y_0 + \int_0^t K(\wb Y(s))
    % \frac{2 a(\wb Y(t)) \nabla q(\wb Y(s))}{q(\wb Y(s))} + b(\wb Y(s))
    \ud s, \quad t \in [0,T]. 
  \end{equation}
  has a unique solution satisfying $\wb Y(t) \in \Theta$ for all $t
  \in (0,T]$. Indeed, let $z(t)$ solve the ODE 
  \[
  z'(t) = 2 a (z(t))\nabla q(z(t)) + q(z(t))b(z(t))
  \]
  for $t \in [0,T]$, with $z(0) = y_0$. For sufficiently small $T$,
  $z(s) \in \Theta$ for $t \in (0,T]$. Hence $q(z(s)) > 0$ for $t \in
  (0,T]$ and the function $F(t) = \int_0^t q(z(s))\,ds$ is
  invertible. Now, it is easy to check that the function $\wb Y(t) =
  z(F^{-1}(t))$ is continuous on $[0,T]$ and satisfies
  \eqref{YbarODE}. Moreover, $\wb Y(t) \in \Theta$ for all $t \in
  (0,T]$. In fact, 
  \[
  \wb Y(t) \sim y_0 + 2 \sqrt{t}\, \frac{a(y_0) \nabla
    q(y_0)}{\avg{\nabla q(y_0),a(y_0) \nabla q(y_0)}^{1/2}}
  \]
  for small $t$.
\end{rmk}

\medskip

We state and prove two properties of the transition path process,
which will be used later.

\begin{prop}\label{prop:Y0contin} 
  Let $F$ be a bounded and continuous functional on
  $C([0,\infty))$. Define
  \[
  g(x) = \wh{\expE}\,[ F(Y) \mid Y_0 = x]
  \]
  where $Y_t$ satisfies \eqref{TPode1}. Then $g \in C(\wb{\Theta})$.
\end{prop}

\begin{proof}
  Suppose that $\{x_n\}_{n=1}^\infty \subset \wb{\Theta}$ and that
  $x_n \to x \in \wb{\Theta}$ as $n \to \infty$.  We claim that there
  must be a subsequence $\{ x_{n_j} \}_{j=1}^\infty$ such that,
  $\mathbb{Q}$-almost surely, 
  \begin{equation} 
    \lim_{j \to \infty} F(Y^j) = F(Y), \label{FYconv} 
  \end{equation}
  where $Y^j_t$ satisfies \eqref{TPode1} with $Y^j_0 = x_{n_j}$, and
  $Y_t$ satisfies \eqref{TPode1} with $Y_0 = x$. Since $F$ is bounded
  and continuous on $C([0,\infty))$, the dominated convergence theorem
  then implies that
  \[
  \lim_{j \to \infty} g(x_{n_j}) = \lim_{j \to \infty} \wh{\expE}\,[
  F(Y) \mid Y_0 = x_{n_j}] = \wh{\expE}\,[ F(Y) \mid Y_0 = x] = g(x).
  \]
  Since the limit is independent of the subsequence, this implies that
  $g(x)$ is continuous.

  To establish \eqref{FYconv}, we must show that $Y^j_t \to Y_t$
  uniformly on compact subsets of $[0,\infty)$.  This follows from
  Corollary~\ref{cor:chebyY12}, as in the proof of
  Theorem~\ref{theo:uniqueYAB}.
\end{proof}

\begin{prop} \label{prop:Yexitbound}
For any $R > 0$, there are constants $k_1, k_2 > 0$ such that 
\[
\mathbb{Q}( Y_t \in \Theta \mid Y_0 = x) \leq k_1 e^{-k_2 t}
\]
holds for all $t \geq 0$ and $x \in \wb{\Theta}$, $|x| < R$.
\end{prop}
\begin{proof}
If $x \in \Theta$, then by the Doob h-transform, we know that
\begin{align*}
  \mathbb{Q}( Y_t \in \Theta \mid Y_0 = x) & = \frac{\Pm(X_s \in
    \Theta \;\forall s \in [0,t], \;\; \tau_B < \tau_A
    \mid X_0 = x)}{\Pm( \tau_B < \tau_A \mid X_0 = x)} \\
  & \leq \frac{\Pm(X_s \in \Theta \;\forall s \in [0,t]\mid X_0 = x)
    \wedge \Pm( \tau_B < \tau_A \mid X_0 = x)}{\Pm( \tau_B < \tau_A
    \mid X_0 = x)} \\
  & = \frac{\Pm(X_s \in \Theta \;\forall s \in [0,t]\mid X_0 = x)
    \wedge q(x)}{q(x)}.
\end{align*}
Since the process $X_t$ is ergodic, there must be constants $C_1, C_2$
such that
\[
\Pm(X_s \in \Theta \;\forall s \in [0,t]\mid X_0 = x) = \Pm(X_s \notin
\wb{A} \cup \wb{B} \;\forall s \in [0,t]\mid X_0 = x) \leq C_1 e^{-C_2
  t}
\]
for all $|x| \leq R$, $t > 0$. So, for any $\epsilon > 0$,
\begin{equation}\label{Qepslower1}
  \mathbb{Q}( Y_t \in \Theta \mid
  Y_0 = x) \leq \frac{C_1 e^{-C_2 t} \wedge \epsilon}{\epsilon}
\end{equation}
holds for all $t > 0$ and $x \in \{ x \in \Theta \mid\; |x| \leq R,\,
q(x) \geq \epsilon \}$.

The bound \eqref{Qepslower1} does not include points near $\partial
A$, where $q(x) < \epsilon$. Fix $\epsilon \in (0,1)$ and define the
set $S = \{ x \in \Theta \mid q(x) < \epsilon \} \cup \wb A$. If
$\epsilon$ is small enough, this set is bounded and we may assume
$\abs{x} < R$ for all $x \in S$. Suppose $Y_0 = x$ with $x \in S \cap
\wb{\Theta}$.  Let $q_t = q(Y_t)$, which satisfies
\[
q_t = q_0 + \int_0^t \frac{|g(Y_s)|^2}{q_s} \ud s + \int_0^t g(Y_s)
\ud \wh{W}_s
\]
where $g(y) = \sqrt{2}(\nabla q(y))^{\TT} \sigma(y)$. By \eqref{hopf}
we know that if $\epsilon > 0$ is small enough, there is a constant
$C_\epsilon > 0$ such that $|g(y)|^2 \geq C_\epsilon$ for all $y \in
\wb{S} \cap \wb{\Theta}$. Therefore, if $Y_t \in \wb{S} \cap
\wb{\Theta}$ for all $t \in [0,T]$, we must have $q_t \leq \epsilon$
for all $t \in [0,T]$ and
\[
q_t \geq \int_0^t \frac{C_\epsilon}{q_s} \ud s + \int_0^t g(Y_s)
\ud\wh{W}_s \geq t\epsilon^{-1} C_{\epsilon} + \int_0^t g(Y_s)
\ud\wh{W}_s
\]
for all $t \in [0,T]$. This happens only if the martingale $M_t =
\int_0^t g(Y_s) \ud\wh{W}_s$ satisfies
\[
M_t \leq \epsilon - t \epsilon^{-1} C_\epsilon, \quad t \in [0,T].
\]
To control the probability of this event, for any $\alpha > 0$, $\beta
> 0$, $T > 0$, Chebychev's inequality implies
\[
\mathbb{Q}(M_T \leq - \alpha T ) \leq e^{-\beta \alpha T}
\wh{\expE}[e^{-\beta M_T}] \leq e^{-\beta \alpha T}
\wh{\expE}\Bigl[\exp\Bigl(\frac{\beta^2}{2} \int_0^T |g|^2
ds\Bigr)\Bigr] \leq e^{- \beta \alpha T + \frac{\beta^2}{2}
  \norm{g}_\infty^2 T}.
\]
By choosing $\beta = \alpha/\norm{g}_\infty^2$ we have $\mathbb{Q}(M_T
\leq - \alpha T ) \leq e^{- \alpha^2 C_3 T}$. Hence there is a
constant $C_4$ such that
\begin{equation}\label{escapeS1}
  \mathbb{Q}\left( Y_t \in \wb{S} \cap \wb{\Theta}, \;\;\; \forall\;t \in [0,T] \mid Y_0 = x \right) \leq e^{- \epsilon^2 C_4 T} 
\end{equation}
holds for all $T > 1$ and $x \in \wb{S} \cap \wb{\Theta}$.

Now we combine \eqref{Qepslower1} and \eqref{escapeS1}. Let $\tau_S =
\inf \{ t > 0 \mid Y_t \in \partial S \}$. By \eqref{escapeS1} we have
$\mathbb{Q} \left( \tau_S > t/2\mid Y_0 = x \right) \leq e^{- C_5 t}$
holds for all $x \in \wb{S} \cap \wb{\Theta}$. Therefore, since
$\tau_S$ is a stopping time, we conclude that 
\begin{align*}
  \mathbb{Q} \left( Y_t \in \Theta \mid Y_0 \in x \right) & \leq
  \mathbb{Q} \left( Y_t \in \Theta, \tau_S < t/2 \mid Y_0 \in x\right)
  + e^{- C_5 t} \\
  & \leq \sup_{y \in \partial S} \mathbb{Q} \left( Y_{t/2} \in \Theta
    \mid Y_0 \in y\right) + e^{- C_5 t} \\
  & \leq \frac{C_1 e^{-C_2 t} \wedge \epsilon}{\epsilon} + e^{-C_5 t}.
\end{align*}
for all $x \in \wb{S} \cap \wb{\Theta}$.
\end{proof}

\begin{proof}[Proof of Theorem \ref{theo:XtauAmin}]
  Since $\tau_{A,n}^+$ is a stopping time, it suffices to prove the
  result for $n = 0$. Fix $\epsilon > 0$ and let $S \supset \wb A$ be
  the open set
  \[
  S = \{ x \in \Theta \mid q(x) < \epsilon \} \cup \wb A.
  \]
  For $\epsilon > 0$ small, this is a bounded set that separates $A$
  and $B$. The boundary $\partial S$ is an isosurface for $q$: $q(x) =
  \epsilon$ for $x \in \partial S$. As $\epsilon \to 0$, $S$ shrinks
  to $A$, and the Hausdorff distance $d_{\mathcal{H}}(\partial
  S, \partial A)$ is $\mathcal{O}(\epsilon)$ (because of
  \eqref{hopf}).

  Recalling that $\tau_{A,0}^+ = \inf \{ t \geq 0 \mid X_t \in \wb{A}
  \}$, we define
  \[
  r_{S,0}  =  \inf \{ t > \tau_{A,0}^+ \mid X_t \in \partial S\}.
  \]
  which is a stopping time with respect to $\mathcal{F}_t$. Then for
  $k \geq 0$, we define inductively the stopping times (see
  Figure~\ref{fig:stoptimes})
  \begin{align*}
    r_{A,k} & = \inf \{ t > r_{S,k}  \mid X_t \in \overline{A} \}, \\
    r_{B,k} & = \inf \{ t > r_{S,k}  \mid X_t \in \overline{B} \},  \\
    r_{S,k+1} & = \inf \{ t > r_{A,k} \mid X_t \in \partial S\}.
  \end{align*}
  Observe that $r_{S,k} < r_{A,k} < r_{S,k+1}$, although it is
  possible that $r_{B,k} = r_{B,k+1}$. Let $r_{AB,k} = r_{A,k} \wedge
  r_{B,k}$, which is finite with probability one.  We also define the
  random time
  \[
  \tau_{S,j} = \inf \{ t > \tau_{A,j}^- \mid X_t \in \partial S \}.
  \]

  \begin{figure}[htp]
    \includegraphics[height = 2in]{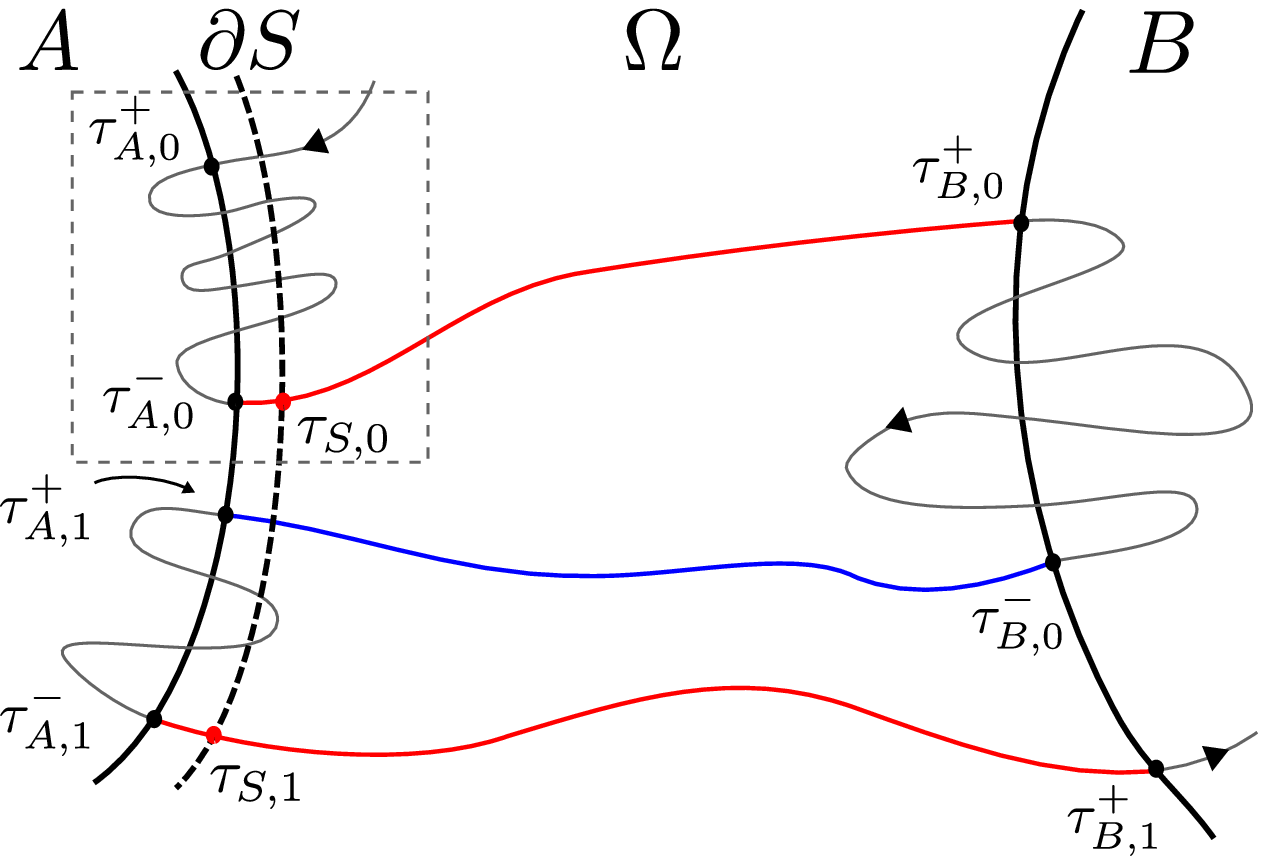} \qquad
    \includegraphics[height = 2in]{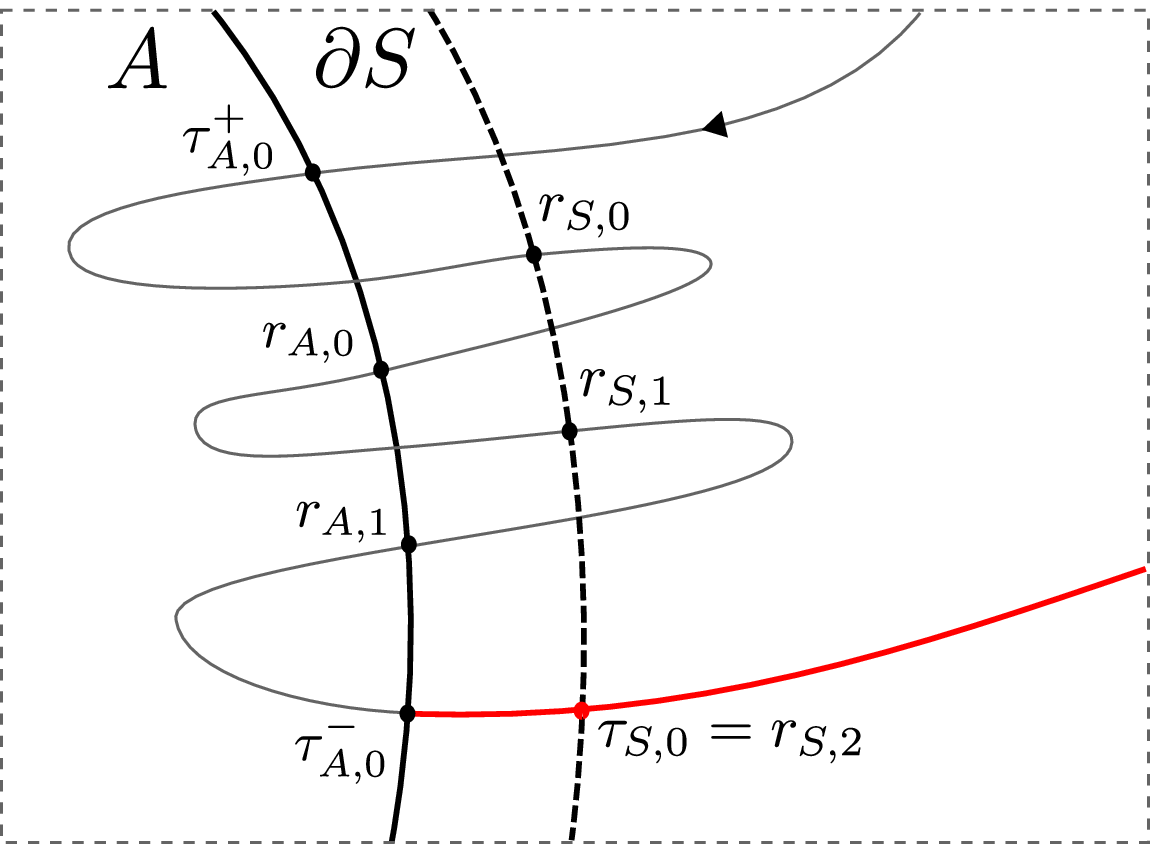}
    \caption{(Left panel) The set $S$ and random times $\tau_{S,
        j}$. (Right panel) Zoom-in of the boxed region together with
      stopping times $r_{S, k}$ and $r_{A, k}$.} \label{fig:stoptimes}
  \end{figure}

  Although $\tau_{S,j}$ is not a stopping time with respect to
  $\mathcal{F}_t$, the relation
  \begin{equation}\label{rsktausj}
    \{ r_{S,k} \mid k \geq 0,\;\; r_{B,k} < r_{A,k} \} = \{ \tau_{S,j}\}_{j=0}^\infty
  \end{equation}
  holds $\Pm$-almost surely. % Let $h_j = \tau_{S,j} - \tau_{A,j}^-$.

  Now, let 
  \[
  Y^0_t = X_{(t + \tau_{A,0}^-) \wedge \tau_{B,0}^+}, \quad \quad t
  \geq 0,
  \]
  and let $h_0 = \tau_{S,0} - \tau_{A,0}^-$. Since $F$ is bounded and
  continuous, and since $h_0 \to 0$ ($\Pm$ almost surely) as $\epsilon
  \to 0$, we have
  \begin{equation} \label{fYthexp2} \expE[F(X_{\cdot \,+
      \tau_{A,0}^-})]= \expE[F(Y^0_{\cdot})] = \lim_{\epsilon \to 0}
    \expE[F(Y^0_{\cdot \, + h_0})].
  \end{equation}
  We will show that
  \[
  \lim_{\epsilon \to 0} \expE[F(Y^0_{\cdot \, + h_0})] = \expE[ g(X_{\tau_{A,0}^-})]
  \]
  where $g(x) = \wh{\expE}[F(Y_\cdot)\mid Y_0 = x]$.

  \medskip

  Let $M$ be the unique (random) integer such that 
  \begin{equation*}
    \tau_{S,0} = r_{S,M}.
  \end{equation*}
  Equivalently, $M = \min \{ k \geq 0 \mid r_{B,k} < r_{A,k}
  \}$. Since $r_{B,k} > r_{A,k}$ for all $k < M$, we have
  \begin{equation}
    F(Y^0_{\cdot \, + h_0}) = \sum_{k=0}^M F(X_{\cdot \, + r_{S,k}})
    \mathbb{I}_{r_{B,k} < r_{A,k}} = \sum_{k=0}^\infty F(X_{\cdot \, +
      r_{S,k}}) \mathbb{I}_{r_{B,k} < r_{A,k}} \mathbb{I}_{k \leq M}.
  \end{equation}
  Observe that the event $\{k \leq M\}$ coincides with the event that
  $r_{B,j} > r_{A,j}$ for all $j < k$, so the event $\{k \leq M\}$ is
  measurable with respect to $\mathcal{F}_{r_{S,k}}$. Therefore, we
  have 
  \begin{equation*}
    \begin{aligned}
      \expE[F(Y^0_{\cdot \, + h_0})] & = \sum_{k=0}^\infty
      \expE\left[F(X_{\cdot \, + r_{S,k}}) \mathbb{I}_{r_{B,k} <
          r_{A,k}} \mathbb{I}_{k \leq M} \right]\\
      & = \sum_{k=0}^\infty \expE \left[\; \expE[F(X_{\cdot \, +
          r_{S,k}}) \mathbb{I}_{r_{B,k} < r_{A,k}} \mathbb{I}_{k \leq
          M} \mid \mathcal{F}_{r_{S,k}}] \; \right] \\
      & = \sum_{k=0}^\infty \expE \left[\mathbb{I}_{k \leq M} \;
        \expE[F(X_{\cdot \, + r_{S,k}}) \mathbb{I}_{r_{B,k} < r_{A,k}}
        \mid  \mathcal{F}_{r_{S,k}}]\; \right] \\
      & = \sum_{k=0}^\infty \expE \left[\mathbb{I}_{k \leq M} \;
        f(X_{r_{S,k}})\; \right], 
      \end{aligned}
    \end{equation*}
    where 
    \begin{equation*} f(x) = \expE[F(X_{\cdot }) \mathbb{I}_{\tau_{B}
        < \tau_{A}} \mid X_0 = x] = q(x) \wh{\expE}[F(Y_{\cdot})\mid
      Y_0 = x].
    \end{equation*}
    The last equality follows from the Doob $h$-transform (since $x
    \in \partial S \subset \Theta$ here).  Since $q(x) = \epsilon$ for
    all $x \in \partial S$, this means 
    \begin{equation} 
      \expE[F(Y^0_{\cdot \, + h_0})] = \epsilon \, \expE
      \left[ \sum_{k=0}^M \; g(X_{r_{S,k}}) \; \right] \label{fexp1}
    \end{equation} 
  where $g(x) = \wh{\expE}[F(Y_{\cdot} )\mid Y_0 = x]$. Note that the
  random integer $M$ depends on $\epsilon$.

  Let $A_j$ denote the event $\{j < M\}$, which occurs if and only if
  $r_{A,k} < r_{B,k}$ for all $k \in \{0,1,\dots,j\}$. Since $q(x) =
  \epsilon$ for all $x \in \partial S$, the event $A_j$ is independent
  of $X_{r_{S,j}} \in \partial S$. Moreover, $P(A_j) = (1 -
  \epsilon)^{j+1}$, since 
  \begin{equation*}
    \begin{aligned}
      \Pm(A_j) & = \expE\left[ \prod_{k=0}^j \mathbb{I}_{r_{A,k} <
          r_{B,k}} \right]  \\
      & = \expE\left[ \; \prod_{k=0}^{j-1} \mathbb{I}_{r_{A,k} <
          r_{B,k}} \; \expE[ \mathbb{I}_{r_{A,j} < r_{B,j}} \mid
        \mathcal{F}_{r_{S,j}}]\; \right] = (1 - \epsilon)
      \Pm(A_{j-1}).
    \end{aligned}
  \end{equation*}
  Similarly, $\Pm( M = j) = \epsilon (1 - \epsilon)^j$.  Now we
  evaluate \eqref{fexp1}:
  \begin{align*}
    \expE[F(Y^0_{\cdot \, + h_0})] & = \epsilon\,\expE[g(X_{r_{S,0}})]
    + \epsilon  \,\expE \left[ \sum_{k=1}^{M}  \; g(X_{r_{S,k}}) \; \right] \\
    & = \epsilon \,\expE[g(X_{r_{S,0}})] + \epsilon\, \expE \left[
      \sum_{j=0}^\infty \mathbb{I}_{A_j} \; g(X_{r_{S,j+1}})]
      \; \right]  \\
    & = \epsilon \,\expE[g(X_{r_{S,0}})] + \epsilon \sum_{j=0}^\infty
    \expE \left[ \mathbb{I}_{A_j} \; g(X_{r_{S,j+1}}) \;
    \right]  \\
    & = \epsilon \,\expE[g(X_{r_{S,0}})] + \epsilon \sum_{j=0}^\infty
    \Pm(A_j)\,\expE \left[\; g(X_{r_{S,j+1}}) \; \right] \\
    & = \epsilon \,\expE[g(X_{r_{S,0}})] + \epsilon \sum_{j=0}^\infty
    (1 - \epsilon)^{j+1} \expE \left[\; g(X_{r_{S,j+1}}) \; \right] \\
    & = \sum_{j=0}^\infty \epsilon (1 - \epsilon)^{j} \, \expE
    \left[\; g(X_{r_{S,j}}) \; \right]  \\
    & = \sum_{j=0}^\infty \Pm( M = j ) \expE\left[\;
      g(X_{r_{S,j}})\;\right] = \expE\left[\;
      g(X_{\tau_{S,0}})\;\right].
  \end{align*}
  Now let $\epsilon \to 0$. Since $g(x)$ is bounded and is continuous
  up to $\partial A$ by Proposition~\ref{prop:Y0contin}, we have (by
  the dominated convergence theorem)
  \begin{equation}
    \lim_{\epsilon \to 0} \expE[ g(X_{\tau_{S,0}})] = \expE[
    \lim_{\epsilon \to 0} g(X_{\tau_{S,0}})] = \expE[
    g(X_{\tau_{A,0}^-})].  
  \end{equation}
\end{proof}

\section{Reactive Exit and Entrance Distributions} \label{sec:REED}

\begin{proof}[Proof of Lemma \ref{lem:sameMass}]
  The equality \eqref{eq:normetaAetaB} is equivalent to 
  \begin{equation*}
    \int_{\partial \Theta} \rho(x) \wh{n}(x)\cdot a(x) \nabla q(x) 
    \ud \sigma_{\Theta}(x) = 0.
  \end{equation*}
  Using \eqref{eq:Lq}, it is then equivalent to 
  \begin{equation*}
    \avg{\rho, Lq} = \avg{L^{\ast} \rho, q} = 0,
  \end{equation*}
  which is obvious.
\end{proof}

Before proving Proposition~\ref{prop:etaB+}, we will need establish
some properties of the entrance and exit distributions and of the
harmonic measure associated with the generator $L$. These results will
also be used later in the paper. First, using integration by parts, we
have
\begin{lem}\label{lem:green}
  Let $D \subset \Rm^d$ be open with smooth boundary. Let $\phi, \psi
  \in C^2(D) \cap C^1(\wb D)$ and bounded. Then
  \begin{multline}\label{eq:green}
    \int_{D} \rho(x) \bigl(\phi(x) L \psi(x) - \psi(x)
    \wt{L}\phi(x) \bigr) \ud x = \int_{\partial D} \rho(x) b\cdot
    \wh{n}(x) \phi(x) \psi(x) \ud \sigma_{D}(x) \\
    + \int_{\partial D} \rho(x) \phi(x) \wh{n}(x) \cdot a \nabla
    \psi(x) - \psi(x) \wh{n}(x) \cdot \divop (a(x) \rho(x) \phi(x))
    \ud \sigma_{D}(x),
  \end{multline}
  where $\wh{n}(x)$ is the exterior normal vector at $x \in \partial
  D$.
\end{lem}

Let us recall some tools from potential theory (see for example the
books \cites{Pinsky:95, Sznitman:98} and also \cites{BoEcGaKl:04,
  BoGaKl:05} where potential theory was applied to analyze diffusion
processes with metastability).  The harmonic measure $H_D(x, \rd y)$
is given by the Poisson kernel corresponding to the boundary value
problem
\begin{equation} \label{ufbvp}
  \begin{cases}
    L u(x) = 0, & x \in D, \\
    u(x) = f(x), & x \in \partial D. 
  \end{cases}
\end{equation}
Therefore, for $f \in C(\partial D)$,
\begin{equation}
  u(x) = \int_{\partial D} H_D(x, \rd y) f(y),
\end{equation}
is the unique solution to \eqref{ufbvp}. Similarly, the harmonic
measure $\wt{H}_D(x, \rd y)$ corresponds to the generator $\wt{L}$
(recall \eqref{LtildeDef}). For the boundary value problem
\begin{equation}
  \begin{cases}
    \wt{L} \wt{u} (x) = 0, & x \in D, \\
    \wt{u}(x) = f(x), & x \in \partial D,
  \end{cases}
\end{equation}
the solution is given by 
\begin{equation}
  \wt{u}(x) = \int_{\partial D} \wt{H}_D(x, \rd y) f(y).
\end{equation}
The harmonic measures have a probabilistic interpretation:
% \cite{Pinsky:95}
$H_D(x, \rd y)$ (resp. $\wt{H}_D(x, \rd y)$) gives the probability
that the process associated with the generator $L$ (resp. $\wt{L}$)
first strikes the boundary $\partial D$ at $\rd y$ after starting at $x$. In
particular,
\[
q(x) = H_D(x,\partial B) \quad \text{and} \quad \wt{q}(x) =
\wt{H}_D(x,\partial A).
\]

We also define the harmonic measures for the conditioned processes as
\begin{equation}
  H_{\Theta}^q(x, \rd y) = \frac{q(y)}{q(x)} H_{\Theta}(x, \rd y).
\end{equation}
For $x \in \Theta$ this is a measure on $\partial B$. For $x
\in \partial A$ where $q(x) = 0$, we may define $H_{\Theta}^q(x,\rd y)$
through a limit: 
\begin{equation}\label{Hplus1}
  H_{\Theta}^q(x,\rd y) = \lim_{\substack{x' \in \Theta \\ x'
      \to x}} \frac{q(y)}{q(x)} H_{\Theta}(x, \rd y) = \frac{
    \wh{n}(x) \cdot a(x) \nabla_x H_{\Theta}(x,dy)}
  {\wh{n}(x) \cdot a(x) \nabla_x q(x)} ,
  \quad x \in \partial A.  
\end{equation}
Recall that $q(y) = 1$ for $y \in \partial B$.

\medskip

Recall the reactive exit and entrance measures $\eta_A^-$, $\eta_A^+$,
$\eta_B^-$ and $\eta_B^+$. They are connected by harmonic measures as
follows:
\begin{prop} \label{prop:etaAetaBmapH}
  \begin{align}\label{etaAB1} 
    & \int_{\partial A} \eta_A^-(\rd x) H_{\Theta}^q(x, \rd y) =
    \eta_B^+(\rd y). \\
    \label{etaAB3} & \int_{\partial A} \eta_A^+(\rd x)
    H_{\wb{B}^C}(x,\rd y) = \eta_B^+(\rd y). \\
    \label{etaAB6} & \int_{\partial B} \eta_B^+(\rd x)
    H_{\wb{A}^C}(x,\rd y) = \eta_A^+(\rd y).
  \end{align}
\end{prop}

\begin{proof}
  We prove \eqref{etaAB1} first. If $f \in C(\partial B)$, let
  $u_f(x)$ solve $Lu = 0$ in $\Theta$ with 
  \begin{equation}\label{ufBC}
    u = \begin{cases}
      f(x), & x \in \partial B,  \\ 
      0, & x \in \partial A.
    \end{cases}
  \end{equation}
  Hence $u(x) \wt{q}(x) = 0$ on $\partial \Theta$. By applying
  \eqref{eq:green} with $\phi(x) = \wt{q}(x)$ and $\psi(x) = u_f(x)$,
  we obtain 
  \begin{equation}\label{fetaB}
    \begin{aligned}
      \int_{\partial A} \rho(x) \wh{n}(x) \cdot a(x) \nabla u_f(x) \ud
      \sigma_{A}(x) & = \int_{\partial B} f(x) \wh{n}(x) \cdot \divop
      (a(x) \rho(x) \wt{q}(x)) \ud \sigma_{B}(x)  \\
      & = \int_{\partial B} f(x) \rho(x) \wh{n}(x) \cdot a(x)
      \nabla \wt{q}(x) \ud \sigma_{B}(x) \\
      & = - \int_{\partial B} f(x) \eta_B^+(dx).
    \end{aligned}
  \end{equation}
  From \eqref{Hplus1} and \eqref{eq:etaA-}, we see that for all $x
  \in \partial A$,
  \[
  \int_{\partial A} \eta_A^-(\rd x) H_{\Theta}^q(x,\rd y) = -
  \int_{\partial A} \rho(x)\wh{n}(x) \cdot a(x) \nabla_x
  H_{\Theta}(x,\rd y) \ud \sigma_{A}(x).
  \]
  Hence for any $f \in C(\partial B)$, we have
  \begin{align*}
    \int_{\partial B} \left( \int_{\partial A} \eta_A^-(\rd x)
      H^q_{\Theta}(x,\rd y)\right) f(y) & = - \int_{\partial B}
    \int_{\partial A} \rho(x)\wh{n}(x) \cdot a(x) \nabla_x \left( f(y)
      H_{\Theta}(x,\rd y)\right) \ud \sigma_{A}(x) \\
    & = - \int_{\partial A} \rho(x)\wh{n}(x) \cdot a(x) \nabla_x
    \left(\int_{\partial B}  H_{\Theta}(x,\rd y)f(y) \right) \rd \sigma_{A}(x) \\
    & = - \int_{\partial A} \rho(x)\wh{n}(x) \cdot a(x) \nabla_x
    u_f(x) \ud x.
  \end{align*}
  Combining this with \eqref{fetaB}, we conclude that
  \[
  \int_{\partial B} \left( \int_{\partial A} \eta_A^-(\rd x)
    H^q_{\Theta}(x,\rd y)\right) f(y) = \int_{\partial B} f(x)
  \eta_B^+(\rd x), \quad \forall \; f \in C(\partial B),
  \]
  which proves \eqref{etaAB1}.

  \smallskip

  To prove \eqref{etaAB3}, let $\psi$ solve $L \psi = 0$ for $x \in
  \wb{B}^C$ with $\psi = f$ on $\partial B$. Then by \eqref{eq:green}
  with $\phi = 1 - \wt{q}$, we have
  \begin{equation*}
    \begin{aligned}
      \int_{\partial A} \eta_A^+(\rd x) \psi(x) & = \int_{\partial A}
      \rho(x) \wh{n}(x) \cdot a(x) \nabla \wt{q}(x) \psi(x)
      \ud \sigma_A(x) \\
      & = - \int_{\partial A} \rho(x) \wh{n}(x) \cdot a(x) \nabla (1 -
      \wt{q}(x)) \psi(x)
      \ud \sigma_A(x) \\
      & = - \int_{\partial A} \psi(x) \wh{n}(x) \cdot \divop( a \rho
      (1 - \wt{q})) \ud \sigma_A(x) \qquad \bigl(\text{since }
      1 - \wt{q} = 0 \text{ on } \partial A \bigr)\\
      & = \int_{\partial B} f \wh{n} \cdot \divop( a \rho (1 -
      \wt{q})) \ud \sigma_B(x) -
      \int_{\partial B} f \rho b \cdot \wh{n}  \ud \sigma_B(x) \\
      & \qquad\qquad - \int_{\partial B} \rho \wh{n} \cdot a \nabla
      \psi \ud \sigma_B(x).
    \end{aligned}
  \end{equation*}
  Applying \eqref{eq:green} with the function $\phi \equiv 1$, we also
  find that
  \begin{equation*}
    0 = - \int_{\partial B} f \wh{n}
    \cdot \divop( a \rho) \ud \sigma_B(x) + \int_{\partial B} f
    \rho b \cdot \wh{n} \ud \sigma_B(x) + \int_{\partial B} \rho
    \wh{n} \cdot a \nabla \psi \ud \sigma_B(x).  
  \end{equation*}
  Therefore, since $1 - \wt{q} = 1$ on $\partial B$, we conclude that
  \begin{equation*}
    \begin{aligned}
      \int_{\partial A} \eta_A^+(\rd x) \psi(x) & = \int_{\partial B}
      f \wh{n} \cdot \divop( a \rho (1 - \wt{q})) \ud \sigma_B(x) -
      \int_{\partial B} f \wh{n}(x) \cdot \divop( a \rho)
      \ud \sigma_B(x) \\
      & = \int_{\partial B} f \rho \wh{n}\cdot a\nabla (1 - \wt{q})\ud
      \sigma_B(x) \\
      & = - \int_{\partial B} f \rho \wh{n}\cdot a\nabla \wt{q}\ud
      \sigma_B(x) = \int_{\partial A} f \eta_B^+(\rd x).
    \end{aligned}
  \end{equation*}
  We arrive at \eqref{etaAB3} noting that
  \begin{equation*}
    \psi(x) = \int_{\partial B} H_{\wb{B}^C}(x, \rd y) f(y). 
  \end{equation*}

  We omit the proof of \eqref{etaAB6} which is analogous to that of
  \eqref{etaAB3} by switching the role of $A$ and $B$.
\end{proof}

By combining \eqref{etaAB3} and \eqref{etaAB6} we immediately obtain
the following:
\begin{cor}\label{cor:invB} 
  Let $P_B(x,\rd y)$ be the probability transition kernel
  \[
  P_B(x,\rd y) = \int_{\partial A} H_{\wb{A}^C}(x,\rd
  z)H_{\wb{B}^C}(z,\rd y), \quad x,y \in \partial B
  \]
  on $\partial B$, and let $P_A(x,\rd y)$ be the probability transition kernel
  \[
  P_A(x,\rd y) = \int_{\partial B} H_{\wb{B}^C}(x,\rd z)
  H_{\wb{A}^C}(z,\rd y), \quad x,y \in \partial A
  \]
  on $\partial A$. Then
  \[
  \int_{x \in \partial B} \eta_B^+(\rd x) P_B(x,\rd y) = \eta_B^+(\rd y).
  \] 
  and
  \[
  \int_{x \in \partial A} \eta_A^+(\rd x) P_A(x,\rd y) = \eta_A^-(\rd y).
  \] 
  That is, $\eta_B^+$ and $\eta_A^+$ are invariant under $P_B$ and
  $P_A$, respectively.
\end{cor}

We are ready to return to the proof of Proposition~\ref{prop:etaB+}. 
\begin{proof}[Proof of Proposition~\ref{prop:etaB+}]
  
  We first verify that $\eta_B^+$ is a probability measure. Taking
  $\psi = q$ and $\phi = \wt{q}$ in \eqref{eq:green}, we obtain using
  the boundary conditions of $q$ and $\wt{q}$ on $\partial A$ and
  $\partial B$,
  \begin{equation*}
    \begin{aligned}
      \eta_{A}^-(\partial A) = \frac{1}{\nu} \int_{\partial_A} \rho
      \wh{n} \cdot a \nabla q \ud \sigma_A & = \frac{1}{\nu}
      \int_{\partial B} \wh{n} \cdot \divop(a\rho \wt{q})
      \ud \sigma_B \\
      & = \frac{1}{\nu} \int_{\partial B} \wh{n} \cdot a \rho \nabla
      \wt{q} \ud \sigma_B = \eta_{B}^+(\partial B). 
    \end{aligned}
  \end{equation*}
  This shows that $\eta_B^+(\partial B) = 1$ and $\nu$ is the correct
  normalization constant.

  \medskip

  Let $g$ be a positive continuous function on $\partial B$. Define
  for $x \not \in \wb{B}$, 
  \begin{equation}
    u(x) = \Em \left[ g(X_{\tau_B})  \mid X_0 = x \right]. 
  \end{equation}
  Hence $u$ satisfies the equation 
  \begin{equation}
    \begin{cases}
      L u(x) = 0, & x \in \wb{B}^c; \\
      u(x) = g(x), & x \in \partial B.
    \end{cases}
  \end{equation}
  Let $H_{\wb{B}^c}(x, \rd y)$ be the harmonic measure (the measure of
  the first hitting point on $\wb{B}$ for the process starting at
  $x$). We have
  \begin{equation}
    u(x) = \int_{\partial B} H_{\wb{B}^c}(x, \rd y) g(y). 
  \end{equation}
  By the maximum principle, $u > 0$ in $\wb{B}^C$. By the Harnack
  inequality and the compactness of $\partial A$, we have
  \begin{equation}
    \sup_{x \in \partial A} u(x) \leq C \inf_{x \in \partial A} u(x)
  \end{equation}
  where the constant $C > 0$ only depends on the elliptic constants of $a$; in
  particular, $C$ is independent of $g$. Therefore, we obtain for any
  $x, x' \in \partial A$, $y \in \partial B$
  \begin{equation}
    0 < C^{-1} \leq \frac{H_{\wb{B}^c}(x, \rd y)}{H_{\wb{B}^c}(x', \rd y)} 
    \leq C < \infty. 
  \end{equation}
  If we define 
  \begin{equation}
    \nu(\rd y) = \inf_{x \in \partial A} H_{\wb{B}^c}(x, \rd y), 
  \end{equation}
  then $\nu(\rd y) > 0$ on $\partial B$ and 
  \begin{equation}\label{eq:minor}
    H_{\wb{B}^c}(x, \rd y) \geq C^{-1} \nu(\rd y)
  \end{equation}
  for any $x \in \partial A$.

  Consider the Markov chain given by $\{X_{\tau_{B, k}^+}\}_{k=0}^\infty$ on $\partial
  B$. Let $P_B$ denote its transition kernel, given by
  \begin{equation}
    P_B(y, \rd y') = \int_{\partial A} H_{\wb{A}^c}(y, \rd x) H_{\wb{B}^c}(x, \rd y'). 
  \end{equation}
  By (\ref{eq:minor}), $P_B$ satisfies Doeblin's condition:
  \begin{equation}
    P_B(y, \rd y') \geq
      C^{-1} \int_{\partial A} H_{\wb{A}^c}(y, \rd x) \nu(\rd y) 
      = C^{-1} \nu(\rd y). 
  \end{equation}
  Therefore, $P_B$ has a unique invariant measure. By
  Corollary~\ref{cor:invB}, this invariant measure is given by
  $\eta_B^+$. The convergence in Proposition~\ref{prop:etaB+} now
  follows (see e.g.~\cite{MeTw:09}).
\end{proof}

\medskip

\begin{proof}[Proof of Theorem \ref{theo:empirical}]
  Consider the family of processes
  \[
  X^{A,n}_t = X_{(t + \tau_{A,n}^+) \wedge \tau_{B,n}^+}.
  \]
  Observe that the $n^{th}$ reactive trajectory $t \mapsto Y^n_t$ is a
  subset of the path $t \mapsto X^{A,n}_t$; specifically, $Y^n_t =
  X^{A,n}_{t + \tau_{A,n}^- - \tau_{A,n}^+}$ for all $t \geq 0$. The
  random sequence of points
  \[
  y_n = X^{A,n}_0 = X_{\tau_{A,n}^+} \in \partial A, \quad n = 0,1,2,\dots
  \]
  corresponds to a Markov chain on the state space $\partial A$ with
  transition kernel
  \[
  P_A(x,\rd y) = \Pm(y_{n+1} \in \rd y \mid y_n = x) = \int_{\partial
    B} H_{\wb{B}^C}(x,\rd z) H_{\wb{A}^C}(z,\rd y).
  \]  
  As shown in the proof of Proposition~\ref{prop:etaB+}, this chain
  has a unique invariant probability distribution $\eta_A^+$ supported
  on $\partial A$:
  \[
  \int_{\partial A} \eta_A^+(\rd x) P_A(x,\rd y) = \eta_A^+(\rd y).
  \]

  The sequence of processes $t \mapsto X^{A,n}_t$ corresponds to a
  Markov chain on the metric space $\mathcal{X} = C([0,\infty))$.  It
  can be shown that this is a Harris chain with unique invariant
  distribution
  \[
  \wb{\mathcal{P}}(U) = \int_{\partial A} \eta_A^+(\rd x)
  \mathcal{P}_x(U), \quad \forall \;U \in \mathcal{B}
  \]
  where $\mathcal{P}_x$ denotes the law on $(\mathcal{X},\mathcal{B})$
  of the process $t \mapsto Z_{t \wedge \tau_{B}}$ where
  \[
  \ud Z_t = b(Z_t) \ud t + \sqrt{2}\, \sigma(Z_t)\ud W_t, \quad Z_0 = x
  \]
  and $\tau_{B}$ is the first hitting time of $Z_t$ to $\wb{B}$. (The
  uniqueness of $\wb{\mathcal{P}}$ follows from the uniqueness of
  $\eta_A^+$ as an invariant distribution for the chain defined by
  transition kernel $P_A$ on $\partial A$.) Therefore (see
  e.g.~\cite{MeTw:09}), for any $\Phi \in
  L^1(\mathcal{X},\mathcal{B},\wb{\mathcal{P}})$ the limit 
  \begin{equation}
    \lim_{N
      \to \infty} \frac{1}{N} \sum_{k=1}^N \Phi(X^{A,k}) = \expE[ \Phi(
    Z_{\cdot \wedge \tau_{B}}) \mid Z_0 \sim \eta_A^+] \label{Philim}
  \end{equation}
  holds $\Pm$-almost surely.

  Using \eqref{Philim} we will establish the following relationship
  between $\eta_A^-$ and $\eta_A^+$:

  %%%%%%%%%%%%%%%%%%%%%%%%%%%%%%%%%%%%%%%%% 
  \begin{lem} \label{lem:etaminusplus} Let $X_t$ satisfy the SDE
    (\ref{XSDE}) with initial distribution $X_0 \sim \eta_A^+$ on
    $\partial A$. Then for any Borel set $U \subset \partial A$,
    \[
    \Pm( X_{\tau_{A,0}^-} \in U \mid X_0 \sim \eta_A^+) =
    \eta_{A}^-(U) = - \frac{1}{\nu} \int_{U} \rho(x) \wh{n}(x) \cdot
    a(x) \nabla q(x) \ud \sigma_A(x).
    \]
  \end{lem}

  \begin{proof}[Proof of Lemma \ref{lem:etaminusplus}.]
    Let $f \in C(\Rm^d)$ be bounded and non-negative. Then by applying
    (\ref{Philim}) to the functional $\Phi(X) = f(X_{\tau_{S,0}^-})$,
    we obtain
    \begin{align*}
      \lim_{\epsilon \to 0} \lim_{N \to \infty} \frac{1}{N}
      \sum_{n=0}^{N-1} f(X_{\tau_{S,n}})
      & =  \lim_{\epsilon \to 0} \expE[ f(X_{\tau_{S,0}}) \mid X_0 \sim \eta_A^+ ]  \\
      & = \expE[ f(X_{\tau_{A,0}^-}) \mid X_0 \sim \eta_A^+ ].
    \end{align*}
    We also have,
    \begin{equation}\label{fYsum}
      \begin{aligned}
        \lim_{N \to \infty} \frac{1}{N} \sum_{n=0}^{N-1}
        f(X_{\tau_{S,n}}) & = \lim_{K \to \infty} \frac{K}{N_K}
        \lim_{K \to \infty} \frac{1}{K} \sum_{k = 0}^{K-1}
        f(X_{r_{S,k}}) \mathbb{I}_{r_{B,k} < r_{A,k}} = \int_{\partial
          S} f(x)\zeta_S(\rd x),
      \end{aligned}
    \end{equation}
    holds $\Pm$-almost surely, where $N_K = |\{ k \in \{0,1,\dots,
    K-1\} \mid r_{B,k} < r_{A,k} \}|$. Here we have used $\zeta_S$ to
    denote the unique invariant distribution (identified below) for
    the Markov chain defined by $X_{r_{S,k}}$ on $\partial
    S$. Therefore,
    \[
    \expE\,[ f(X_{\tau_{A,0}^-}) \mid X_0 \sim \eta_A^+ ] =
    \lim_{\epsilon \to 0} \int_{\partial S} f(x) \zeta_S(\rd x).
    \]

    We claim that if $f(x)$ is uniformly continuous in a neighborhood
    of $\partial A$, then
    \begin{equation}\label{StoA2}
      \lim_{\epsilon \to 0} \int_{\partial S} \zeta_S(\rd x)  
      f(x) =  \int_{\partial A} \eta_A^-(\rd x) f(x).
    \end{equation}
    First, let us identify the invariant distribution $\zeta_S$. By
    applying Corollary~\ref{cor:invB} (replacing $B$ by $\wb{S}^C$) we
    can identify $\zeta_S$ as
    \[
    \zeta_S(\rd x)\; ( = \eta_S^+(\rd x) ) = - \frac{\epsilon}{\nu}
    \rho(x) \wh{n}(x) \cdot a(x) \nabla \wt{q}_S(x) \ud \sigma_S(x),
    \]
    where $\wh{n}(x)$ is the exterior normal at $x \in \partial S$,
    and $\wt{q}_S$ satisfies $\wt{L} \wt{q}_S = 0$ in $S$ with
    \[
    \wt{q}_S(x) = 
    \begin{cases}
      1, & x \in \partial A \\
      0, & x \in \partial S.
    \end{cases}
    \]
    Note that $\nu$ is independent of $\epsilon$. Let $\delta >
    \epsilon$ be small, and suppose that $f(x)$ is continuous on the
    closed set $\{ x \in \wb{\Theta} \mid 0 \leq q(x) \leq \delta
    \}$. (This set contains both $\partial A$ and $\partial S$). A
    computation similar to \eqref{fetaB} (replacing $B$ by $S$) shows
    that for any such function, we have 
    \begin{equation} 
      \int_{\partial S} \zeta_S(\rd x) f(x) = -
      \frac{\epsilon}{\nu}\int_{\partial A} \rho(x) \wh{n}(x) \cdot
      a(x) \nabla u_{f,S}(x) \ud \sigma_A(x), \label{ufSint}
    \end{equation}
    where $u_{f,S}$ satisfies $L u = 0$ in $S \setminus \wb{A}$, and
    \[
    u_{f,S}(x) =
    \begin{cases}
      f(x), & x \in \partial S \\
      0, & x \in \partial A.
    \end{cases}
    \]
    Since $f \geq 0$, we have $u > 0$ in $S \setminus \wb{A}$. Now,
    let us define
    \[
    z_{f,S}(x) = \epsilon \frac{u_{f,S}(x)}{q(x)}, \quad x \in \wb{S}
    \setminus A,
    \]
    which satisfies $L^q z = 0$ in $S \setminus \wb{A}$, with $z = f$
    on $\partial S$ (recall that $q(x) = \epsilon$ for all $x
    \in \partial S$).  By the boundary Harnack inequality
    (\cites{Bau:84, CS:05}), $z_{f,S}(x)$ is bounded and
    H\"older continuous on $\wb{S} \setminus A$ (including $\partial
    A$). We claim that for any $x_0 \in \partial A$, we have
    \begin{equation}\label{zgradlim}
      \lim_{x \to x_0} \nabla u_{f,S}(x) = \epsilon^{-1} z_{f,S}(x_0) \nabla q(x_0). 
    \end{equation}
    Since $\nabla u_{f,S}$, $\nabla q$, and $z_{f,S}$ are continuous
    up to $\partial A$, this is true if and only if
    \[
    \lim_{x \to x_0} q(x) \nabla z_{f,S}(x) = 0.
    \]
    Suppose $q(x) \nabla z_{f,S}(x) \to v \neq 0$ as $x \to x_0
    \in \partial A$. Then we must have
    \[
    \lim_{x \to x_0} \nabla u_{f,S}(x) - z_{f,S}(x) \nabla q(x) = v
    \]
    so that $v$ must be a multiple of $\wh{n}(x_0)$ (since $u$ and $q$
    vanish on $\partial A$).  Thus, we would have 
    \begin{equation}
      \wh{n}(x_0)\cdot
      \nabla z_{f,S}(x) \sim (\wh{n}(x_0) \cdot v)
      q(x)^{-1} \label{zblowup} 
    \end{equation} 
    as $x \to x_0 \in \partial A$. If $v \neq 0$, then $(\wh{n}(x_0)
    \cdot v) \neq 0$, so (\ref{zblowup}) and the fact that $q = 0$ on
    $\partial A$ would contradict the boundedness of
    $z_{f,S}(x)$. Therefore, \eqref{zgradlim} must hold.

    Combining \eqref{ufSint} and \eqref{zgradlim} we obtain
    \[
    \int_{\partial S} \zeta_S(\rd x) f(x) = -
    \frac{1}{\nu}\int_{\partial A} \rho(x) \wh{n}(x) \cdot a(x) \nabla
    q(x) z_{f, S}(x) \ud \sigma_A(x) = \int_{\partial A} \eta_A^-(\rd
    x) z_{f,S}(x).
    \]
    % This $z_{f,S}$ satisfies $L^+ z = 0$ in $S \setminus \wb{A}$,
    % and
    Therefore, as $\epsilon \to 0$,
    \begin{equation}\label{Slimf}
      \lim_{\epsilon \to 0} \int_{\partial S} \zeta_S(\rd x) f(x) =
      \lim_{\epsilon \to 0} \int_{\partial A} \eta_A^-(\rd x) z_{f,S}(x) =
      \int_{\partial A} \eta_A^-(\rd x) f(x).   
    \end{equation} 
    This establishes (\ref{StoA2}) and completes the proof of Lemma
    \ref{lem:etaminusplus}.
  \end{proof}

  Now we continue with the proof of Theorem \ref{theo:empirical}. We
  will apply Theorem \ref{theo:XtauAmin}. Suppose that $F \in
  L^1(\mathcal{X},\mathcal{B},\mathcal{Q}_{\eta_A^-})$, and define the
  functional
  \[
  \Phi(X) = F(X_{(\cdot + \tau_{A,0}^-) \wedge \tau_{B,0}^+}).
  \]
  Combining Theorem \ref{theo:XtauAmin} and Lemma
  \ref{lem:etaminusplus} we see that $\Phi \in
  L^1(\mathcal{X},\mathcal{B},\wb{\mathcal{P}})$, since
  \begin{align*}
    \wb{\mathcal{P}}( \Phi(X) > \alpha) & = \Pm( \Phi(X) > \alpha
    \mid X_0 \sim \eta_A^+) \\
    & = \Pm( F(X_{(\cdot\, + \tau_{A,0}^-) \wedge \tau_{B_0}^+}) >
    \alpha
    \mid X_0 \sim \eta_A^+) \\
    & = \mathbb{Q}( F(Y) > \alpha \mid Y_0 \sim \eta_A^-) =
    \mathcal{Q}( F(Y) > \alpha).
  \end{align*}
  Therefore, 
  \begin{equation*}
    \frac{1}{N} \sum_{k = 0}^{N-1} F(Y^k)  =  
    \frac{1}{N} \sum_{k = 0}^{N-1} F(X^{A,k}_{(\cdot + \tau_{A,k}^-) \wedge \tau_{B,k}^+})  
    = \frac{1}{N} \sum_{k = 0}^{N-1} \Phi(X^{A,k}_{\cdot}).  
  \end{equation*}
  By \eqref{Philim} and Theorem~\ref{theo:XtauAmin}, we now conclude
  that the limit
  \begin{align*}
    \lim_{N \to \infty} \frac{1}{N} \sum_{k = 0}^{N-1} F(Y^k) = \expE[
    \Phi(Z_{\cdot \wedge \tau_{B}}) \mid Z_0 \sim \eta_A^+] =
    \wh{\expE}[ F(Y) \mid Y_0 \sim \eta_A^-]
  \end{align*}
  holds $\Pm$-almost surely. This completes the proof of Theorem
  \ref{theo:empirical}.
\end{proof}

\section{Reaction rate, density and current of transition
  paths} \label{sec:appltpt}

\subsection{Reaction rate}

\begin{proof}[Proof of Proposition \ref{prop:reactrate}]
  Denote $\tau_B$ the first hitting time of $X_t$ to $\wb{B}$.
  Consider the mean first hitting time
  \begin{equation*}
    u_B(x) = \mathbb{E}\left[\tau_B \mid X_0 = x\right], 
  \end{equation*}
  which satisfies the equation 
  \begin{equation}\label{eq:defuB}
    \begin{cases}
      L u_B(x) = - 1, & x \in \Theta \\
      u_B(x) = 0, & x \in \partial B.
    \end{cases}
  \end{equation}

  By definition of $\eta_A^+$, we have
  \begin{equation}\label{eq:aveub}
    \begin{aligned}
      \int_{\partial A} \eta_A^+(\rd x) u_B(x) & = \frac{1}{\nu}
      \int_{\partial A} \rho(x) u_B(x) \wh{n}(x) \cdot a(x) \nabla
      \wt{q}(x) \ud \sigma_A(x). 
    \end{aligned}
  \end{equation}  
  
  Observe that 
  \begin{equation*}
    \begin{aligned}
      \int_{\RR^d} \rho(x) \wt{q}(x) \ud x &
      = \int_{B^c} \rho(x) \wt{q}(x) \ud x \\
      & \stackrel{\eqref{eq:defuB}}{=} - \int_{B^c} \rho(x) \wt{q}(x)
      (L u_B)(x) \ud x \\
      & = - \int_A \rho(x) (Lu_B)(x) \ud x - \int_{\Theta}
      \rho(x) \wt{q}(x) (L u_B)(x) \ud x.
    \end{aligned}
  \end{equation*}
  Using \eqref{eq:green} with $D = A$, $\phi(x) = 1$ and $\psi(x) =
  u_B$, we obtain
  \begin{equation*}
    \int_A \rho (L u_B) \ud x = - \int_{\partial A} \rho
    b \cdot \wh{n} u_B \ud \sigma_A(x) - \int_{\partial A} \rho \wh{n} \cdot a 
    \nabla u_B \ud \sigma_A(x)  + \int_{\partial A} u_B \wh{n} \cdot 
    \divop(a \rho)     \ud \sigma_A(x),
  \end{equation*}
  where $\wh{n}$ is the interior normal vector at $\partial A$. Apply
  \eqref{eq:green} again with $D = \Theta$, $\phi = \wt{q}$ and $\psi
  = u_B$, 
  \begin{equation*}
    \int_{\Theta} \rho \wt{q} (L u_B) \ud x = \int_{\partial A} \rho b \cdot 
    \wh{n} u_B \ud \sigma_A(x) + \int_{\partial A} \rho \wh{n} \cdot a 
    \nabla u_B \ud \sigma_A(x) - \int_{\partial A} u_B \wh{n} \cdot 
    \divop(a \rho \wt{q}) \ud \sigma_A(x). 
  \end{equation*}
  Combining the two with \eqref{eq:aveub}, we get 
  \begin{equation*}
    \int_{\partial A} \eta_A^+(\rd x) u_B(x) = \frac{1}{\nu} 
    \int_{\partial A} \rho u_B \wh{n} \cdot 
    a \nabla \wt{q} \ud \sigma_A(x) = \frac{1}{\nu} 
    \int_{\RR^d} \rho \wt{q} \ud x. 
  \end{equation*}
  Similarly, defining $u_A(x)$ to be the mean first hitting time of $X_t$ to
  $\wb{A}$ starting at $x$, we have
  \begin{equation*}
    \int_{\partial B} \eta_B^+(\rd x) u_A(x) = \frac{1}{\nu} 
    \int_{\RR^d} \rho (1 - \wt{q}) \ud x.
  \end{equation*}
  Add the integrals together to obtain
  \begin{equation*}
    \int_{\partial A} \eta_A^+(\rd x) u_B(x) 
    + \int_{\partial B} \eta_B^+(\rd x) u_A(x) 
    = \frac{1}{\nu}.
  \end{equation*}
  
  \medskip

  On the other hand, observe that 
  \begin{equation*}
    \begin{aligned}
      \frac{1}{\nu_R} &= \lim_{N_T \to \infty} \frac{T}{N_T} \\
      &= \lim_{N \to \infty} \frac{1}{N} \sum_{n=0}^{N-1} ( \tau_{A,
        n+1}^+ - \tau_{A, n}^+) \\
      &= \lim_{N \to \infty} \frac{1}{N} \sum_{n=0}^{N-1} ( \tau_{B,
        n}^+ - \tau_{A, n}^+) + \lim_{N\to\infty} \frac{1}{N}
      \sum_{n=0}^{N-1} ( \tau_{A, n+1}^+ - \tau_{B, n}^+).
    \end{aligned}
  \end{equation*}
  As $N \to \infty$, we have 
  \begin{equation*}
    T_{AB} = \lim_{N \to \infty} \frac{1}{N} \sum_{n=0}^{N-1} ( \tau_{B,
      n}^+ - \tau_{A, n}^+) = \mathbb{E} [ \tau_B \mid 
    X_0 \sim \eta_A^+ ] = \int_{\partial A} \eta_A^+(\rd x) u_B(x), 
  \end{equation*}
  and similarly 
  \begin{equation*}
    T_{BA} = \lim_{N\to\infty} \frac{1}{N}
    \sum_{n=0}^{N-1} ( \tau_{A, n+1}^+ - \tau_{B, n}^+) = \int_{\partial B}
    \eta_B^+(\rd x) u_A(x). 
  \end{equation*}
  Therefore 
  \begin{equation*}
    \frac{1}{\nu_R} = \int_{\partial A} \eta_A^+(\rd x) u_B(x) + \int_{\partial B}
    \eta_B^+(\rd x) u_A(x) = \frac{1}{\nu}, 
  \end{equation*}
  or equivalently $\nu = \nu_R$.

  From Theorem \ref{theo:empirical} it follows immediately that 
  \[
  C_{AB} = \int_{\partial A} \eta_A^-(\rd x) v_B(x).
  \]
  Indeed, the functional $F: Y \to \tau^Y_B$ is in
  $L^1(\mathcal{X},\mathcal{B},\mathcal{Q}_{\eta_A^-})$ by Proposition
  \ref{prop:Yexitbound}. The function $v_B(x) = \wh{\expE}[ \tau_B^Y
  \mid Y_0 = x]$ satisfies
  \[
  L^q v_B = -1, \quad x \in \Theta
  \]
  with $v(x) = 0$ for $x \in \partial B$. Hence, the function $w(x) =
  q(x) v_B(x)$ satisfies $L w = -q$ for $x \in \Theta$ with boundary
  condition $w(x) = 0$ for $x \in \partial \Theta$. Moreover, for $x_0
  \in \partial A$, we have
  \[
  v_B(x_0) = \lim_{x \to x_0} \frac{w(x)}{q(x)} = \frac{\wh{n}(x_0)
    \cdot a(x_0) \nabla w(x_0)}{ \wh{n}(x_0) \cdot a(x_0) \nabla
    q(x_0)}.
  \]
  Therefore,
  \[
  \int_{\partial A} \eta_A^-(\rd x) v_B(x) = -
  \frac{1}{\nu}\int_{\partial A} \rho(x) \wh{n}(x) \cdot a(x) \nabla
  w(x) \ud \sigma_A(x).
  \]
  Now applying \eqref{eq:green} with $D = \Theta$, $\phi = \tilde q$
  and $\psi = w$, we have
  \[
  - \frac{1}{\nu}\int_{\partial A} \rho(x) \wh{n}(x) \cdot a(x) \nabla
  w(x) \ud \sigma_A(x) = \frac{1}{\nu} \int_{\Theta} \rho(x) \tilde
  q(x)q(x) \ud x.
  \]

  It remains to show that 
  \begin{equation*}
    \nu = \int_{\RR^d} \rho \nabla q \cdot a \nabla q \ud x. 
  \end{equation*}
  Using integration by parts, we have 
  \begin{equation*}
    \begin{aligned}
      \int_{\RR^d} \rho \nabla q \cdot a \nabla q  \ud x & =
      \int_{\Theta} \rho \nabla \bigl(q -
      \frac{1}{2}\bigr) \cdot a \nabla q  \ud x \\
      & = - \int_{\Theta} \nabla\cdot (\rho a \nabla q) \bigl(q -
      \frac{1}{2}\bigr) \ud x + \int_{\partial A} \rho \bigl(q -
      \frac{1}{2}\bigr)  \wh{n} \cdot a \nabla q \ud \sigma_A(x) \\
      & \hspace{13em} + \int_{\partial B} \rho \bigl(q -
      \frac{1}{2}\bigr)  \wh{n} \cdot a \nabla q \ud \sigma_B(x).
    \end{aligned}
  \end{equation*}
  The first term on the right hand side vanishes as 
  \begin{equation*}
    \begin{aligned}
      \int_{\Theta} \nabla\cdot (\rho a \nabla q) \bigl(q -
      \frac{1}{2}\bigr) \ud x & = \int_{\Theta} \Bigl(\rho \tr a
      \nabla^2 q + \rho b \cdot \nabla q\Bigr) \bigl(q -
      \frac{1}{2}\bigr) \ud x  \\
      & \quad + \frac{1}{2} \int_{\Theta} \Bigl( \divop(\rho a) \cdot
      \nabla - \rho b \nabla \Bigr) (q^2 - q) \ud x \\
      & = \int_{\Theta} \rho (L q) \bigl(q - \frac{1}{2}\bigr) \ud x -
      \frac{1}{2} \int_{\Theta} (L^{\ast} \rho)(q^2 - q) = 0,
    \end{aligned}
  \end{equation*}
  where we have used that $q^2 - q = 0$ on $\partial A \cup \partial B$.
  The conclusion then follows from Lemma~\ref{lem:sameMass}, $q = 0$ on
  $\partial A$, and $q = 1$ on $\partial B$.
\end{proof}

\subsection{Density of transition paths}

We define the Green's function $G_{\Theta}$ of the operator $L$ in
$\Theta$ with Dirichlet boundary condition on $\partial \Theta$:
\begin{equation} \label{GreenThetaDef}
  \begin{cases}
    L G_{\Theta}(x, y) = - \delta_y(x), & x \in \Theta, \\
    G_{\Theta}(x, y) = 0, & x \in \partial \Theta.
  \end{cases}
\end{equation}
The existence of the Green's function is guaranteed by the ergodicity
of $X_t$ in $\Rm^d$, which implies that $X_t$ is transient in $\Theta$
(see e.g.~\cite{Pinsky:95}*{Section 4.2}).

\begin{lem}\label{lem:GreenFcn}
  Let $G_{\Theta}$ be the Green's function of $L$ in $\Theta$ with
  Dirichlet boundary condition on $\partial \Theta$.  We have
  \begin{equation}\label{eq:Ginterior}
    G_{\Theta}^{q}(x, y) \equiv \int_0^{\infty} Q_R(t, x, y) \ud t 
    = \frac{q(y) G_{\Theta}(x, y) }{q(x)}.
  \end{equation}
  In particular, for $x \in \partial A$, $y \in \Theta$
  \begin{equation}\label{eq:Gboundary}
    G_{\Theta}^q(x, y) = \frac{ q(y) \wh{n}(x) \cdot a(x) 
      \nabla_x G_{\Theta}(x, y)}
    { \wh{n}(x) \cdot a(x) \nabla q(x)}.
  \end{equation}
\end{lem}

\begin{proof}
  Fix $y \in \Theta$. For $x \in \Theta$, \eqref{eq:Ginterior} follows
  from \cite{Pinsky:95}*{Proposition 4.2.2}. Specifically, the
  function $G_{\Theta}^{q}(x, y)$ defined by
  \begin{equation*}
    G_{\Theta}^{q}(x, y) = \int_0^{\infty} Q_R(t, x, y) \ud t 
  \end{equation*}
  is related to the Green's function \eqref{GreenThetaDef} by the
  formula 
  \begin{equation*}
    G_{\Theta}^{q}(x, y) = \frac{q(y) G_{\Theta}(x, y) }{q(x)}, 
    \quad x,y \in \Theta.
  \end{equation*}
  Because of the regularity of the coefficients $a(x)$ and $b(x)$,
  Schauder-type interior and boundary estimates imply that $G(\cdot,y)
  \in C^{2,\alpha}(\wb{\Theta} \setminus \{y\})$. Since $G(x,y) = q(x)
  = 0$ for $x \in \partial A$, the Hopf Lemma implies that for all $x
  \in \partial A$, $\nabla_x G(x,y)$ is a nonzero multiple of
  $\wh{n}(x)$. That is, for all $x \in \partial A$, $\nabla_x G(x,y) =
  r(x)\wh{n}(x)$ for some continuous $r(x) < 0$. The same is true for
  $q$. Therefore, $G_{\Theta}^{q}(x, y)$ is continuous in $x$ up to
  the boundary $\partial \Theta$ and for $x_0 \in \partial A$,
  \[
  \lim_{x \to x_0,\, x \in \Theta} G_{\Theta}^{q}(x, y) = \frac{ q(y)
    \wh{n}(x_0) \cdot a(x_0) \nabla_x G_{\Theta}(x_0, y)} {
    \wh{n}(x_0) \cdot a(x_0) \nabla q(x_0)}.
  \]

  It remains to show that for $x_0 \in \partial A$,
  \begin{equation}\label{Gqrep1}
    \frac{ q(y) \wh{n}(x_0) \cdot a(x_0) 
      \nabla_x G_{\Theta}(x_0, y)}
    { \wh{n}(x_0) \cdot a(x_0) \nabla q(x_0)} 
    = \int_0^{\infty} Q_R(t, x_0, y) \ud t. 
  \end{equation}
  Let $\varphi \geq 0$ be smooth and compactly supported in
  $\Theta$. By Proposition~\ref{prop:Y0contin}, we have
  \[
  \lim_{x \to x_0} \wh{\expE}[ \varphi(Y_t)\mid Y_0 = x] = \wh{\expE}[
  \varphi(Y_t)\mid Y_0 = x_0].
  \]
  Moreover,
  \[
  \wh{\expE}[ \varphi(Y_t)\mid Y_0 = x] \leq \norm{\varphi}_\infty
  \mathbb{Q}( Y_t \in \Theta\mid Y_0 = x).
  \]
  By Proposition \ref{prop:Yexitbound}, for any $R > 0$, there are
  constants $k_1, k_2 > 0$ such that $\mathbb{Q}( Y_t \in \Theta\mid
  Y_0 = x) \leq k_1e^{-k_2 t}$ for all $x \in \theta$, $|x| < R$, $t
  \geq 0$. Therefore, we have $\wh{\expE}[ \varphi(Y_t)\mid Y_0 = x]
  \leq \norm{\varphi}_\infty k_1 e^{-k_2 t}$ so the dominated
  convergence theorem implies that
  \begin{equation}\label{Gconv1}
    \begin{aligned}
      \lim_{x \to x_0} \int_{\Theta} G_{\Theta}^{q}(x, y) \varphi(y)
      \ud y & = \lim_{x \to x_0} \int_0^\infty \wh{\expE}[
      \varphi(Y_t)\mid Y_0 = x] \ud t  \\
      & = \int_0^\infty \wh{\expE}[ \varphi(Y_t)\mid Y_0 = x_0] \ud t  \\
      & = \int_0^\infty \left( \int_{\Theta} Q(t,x_0,y) \varphi(y) \ud
        y \right) \ud t.
    \end{aligned}
  \end{equation}
  On the other hand, we also have
  \begin{equation}\label{Gconv2}
    \lim_{x \to x_0} \int_{\Theta} G_{\Theta}^{q}(x, y) \varphi(y) \ud y = 
    \int_{\Theta} \frac{ q(y) \wh{n}(x_0) \cdot a(x_0) 
      \nabla_x G_{\Theta}(x_0, y)}
    { \wh{n}(x_0) \cdot a(x_0) \nabla q(x_0)} \varphi(y) \ud y. 
  \end{equation}
  Therefore, by combining \eqref{Gconv1} and \eqref{Gconv2} we
  conclude
  \begin{align*}
    \int_{\Theta} \frac{ q(y) \wh{n}(x_0) \cdot a(x_0) \nabla_x
      G_{\Theta}(x_0, y)} { \wh{n}(x_0) \cdot a(x_0) \nabla q(x_0)}
    \varphi(y) \ud y
    & = \int_0^\infty \int_{\Theta} Q(t,x_0,y) \varphi(y) \ud y \ud t \\
    & = \int_{\Theta} \left( \int_0^\infty Q(t,x_0,y) \ud t \right)
    \varphi(y) \ud y.
  \end{align*}
  Since $\varphi$ is arbitrary, this implies \eqref{Gqrep1}.
\end{proof}

\begin{proof}[Proof of Proposition \ref{prop:rhoR}]
  Using Lemma~\ref{lem:GreenFcn} and (\ref{rhoRPR}), 
  \begin{equation}\label{eq:rhoRgreen}
    \rho_R(z) = \nu_R \int_{\partial A} \eta_A^-(\rd x) G^q_{\Theta}(x, z). 
  \end{equation}
  Recall the explicit formula of $\eta_A^-$ in terms of $q$
  \eqref{eq:etaA-}, we obtain for $z \in \Theta$
  \begin{equation*}
    \begin{aligned}
      \rho_R(z) & = -\int_{\partial A} \rho(x) \frac{ q(y) \wh{n}(x)
        \cdot a \nabla_x G_{\Theta}(x, z)} { \wh{n}(x) \cdot a \nabla
        q(x)}
      \wh{n}(x) \cdot a \nabla q(x) \ud \sigma_A(x) \\
      & = -q(y) \int_{\partial A} \rho(x) \wh{n}(x) \cdot a \nabla_x
      G_{\Theta}(x, z) \ud \sigma_A(x).
    \end{aligned}
  \end{equation*}
  Apply \eqref{eq:green} by taking $\psi(x) = G_{\Theta}(x, y)$ and
  $\phi(x) = \wt{q}(x)$, we conclude that 
  \begin{equation*}
    \begin{aligned}
      \rho_R(y) & = - q(y) \int_{\partial \Theta} \rho(x)
      \phi(x) \wh{n}(x) \cdot a \nabla \psi(x) \ud \sigma_{\Theta}(x) \\
      & = - q(y) \int_{\Theta} \rho(x) \phi(x) L \psi(x) \\
      & = \rho(y) q(y) \wt{q}(y). 
    \end{aligned}
  \end{equation*}
  Here to get the second equality, we have used that $\wt{L}\wt{q} =
  0$ in $\Theta$ and $\psi(x) = 0$ on $\partial \Theta$.
\end{proof}

\subsection{Current of transition paths}

\begin{proof}[Proof of Proposition~\ref{prop:Jboundary}]
  It follows from a direct calculation from the definition of $J_R$ as
  \eqref{eq:JR}, noticing that $q = 0, \wt{q} = 1$ on $\partial A$,
  and $q = 1, \wt{q} = 0$ on $\partial B$.
\end{proof}

\begin{proof}[Proof of Corollary~\ref{cor:SJR}]
  By Proposition~\ref{prop:Jboundary}, we have
  \begin{equation*}
    \nu_R = - \int_{\partial A} \wh{n}(x) \cdot J_R(x) \ud \sigma_A(x).
  \end{equation*}
  Hence, it suffices to show that 
  \begin{equation*}
    \int_{\partial A} \wh{n}(x) \cdot J_R(x) \ud \sigma_A(x)
    + \int_{\partial S} \wh{n}(x) \cdot J_R(x) \ud \sigma_S(x) = 0,
  \end{equation*}
  which follows from the fact that $J_R$ is divergence free in
  $\Theta$ (see \eqref{eq:divJR}).
\end{proof}

\begin{proof}[Proof of Proposition~\ref{prop:JR}]

  Using Proposition~\ref{prop:Jboundary} for the left hand side of
  \eqref{eq:defJR}, we obtain
  \begin{equation*}
    \int_{\partial B} f(x) \eta_B^+(\rd x) - \int_{\partial A} f(x)
    \eta_A^-(\rd x) = \frac{1}{\nu_R} \int_{\partial B} f \wh{n} \cdot J_R \ud
    \sigma_B + \frac{1}{\nu_R} \int_{\partial A} f \wh{n} \cdot J_R \ud \sigma_A,
  \end{equation*}
  where $\wh{n}$ is the unit normal exterior to $\Theta$. Equation
  \eqref{eq:defJR} then follows from the divergence theorem.

  Now fix any $g \in C^1(\partial B)$, we extend $g$ to $\wb{\Theta}$
  using the flow \eqref{eq:Z}: for any $x \in \wb{\Theta}$, we define
  \begin{equation}\label{eq:defg}
    g(x) = g(Z_{t_B}^x), \quad \text{with } Z_0^x = x. 
  \end{equation}
  In particular, for $x \in \partial A$, we have $
  g(x) = g(\Phi_{J_R}(x))$, in other words, 
  \begin{equation}\label{eq:pullback}
    g \vert_{\partial A} = \Phi_{J_R}^{\ast}  (g \vert_{\partial B}).
  \end{equation}
  By the construction \eqref{eq:defg}, for any $x \in \Theta$, $J_R
  \cdot \nabla g = 0$. Combining with the first part of the
  Proposition and \eqref{eq:pullback}, we obtain
  \begin{equation*}
    \int_{\partial B} g(x) \eta_B^{+}(\rd x) = 
    \int_{\partial A} \Phi_{J_R}^{\ast}  g \, \eta_A^-(\rd x). 
  \end{equation*}
  Therefore, $\Phi_{J_R, \ast}(\eta_A^-) = \eta_B^+$.
\end{proof}

%%%%%%%%%%%%%%%%%%%%%%%%%%%%%%%%%%%%%%%%%%%%%%%%%%%%%%%%%%%
%%%%%%%%%%%%%%%%%%%%%%%%%%%%%%%%%%%%%%%%%%%%%%%%%%%%%%%%%%%

\bibliographystyle{amsxport} 
\bibliography{tpt}

\end{document}